\documentclass[11pt]{amsart}

\usepackage{amsmath, amsfonts, amssymb, amsthm}
\usepackage{color}
\usepackage[left = 1.5cm, right = 1.5cm, bottom = 3cm]{geometry}
\usepackage{enumitem}
\usepackage{tikz}
\usepackage{mathtools}
\usepackage{float}
\usepackage{color}

\newcommand*\mR{\mathbb{R}}

\newcommand*\intom{\int_\Omega}
\newcommand*\fpqom{F_{p,q}(\Omega)}
\newcommand*\mW{\mathcal{W}}
\newcommand*\mN{\mathcal{N}}
\newcommand*\eps{\varepsilon}
\newcommand*\AJ{\textnormal{\textbf{A1}}}
\newcommand*\AD{\textnormal{\textbf{A2}}}
\newcommand*\AT{\textnormal{\textbf{A3}}}
\newcommand*\AAA{\textnormal{(\textbf{A})}}
\newcommand*\BB{\textnormal{(\textbf{B})}}
\newcommand*\CC{\textnormal{(\textbf{C})}}
\newcommand*\CJ{\textnormal{(\textbf{C1})}}
\newcommand*\CD{\textnormal{(\textbf{C2})}}
\newcommand*\red{}
\DeclareMathOperator{\intr}{int}
\DeclareMathOperator{\diam}{diam}
\DeclareMathOperator{\Sh}{\textnormal{\textbf{Sh}}}
\DeclareMathOperator{\SH}{\textnormal{\textbf{SH}}}

\theoremstyle{plain}
\newtheorem{proposition}{Proposition}[section]
\newtheorem{lemma}[proposition]{Lemma}
\newtheorem{theorem}[proposition]{Theorem}

\theoremstyle{definition}
\newtheorem{definition}[proposition]{Definition}
\newtheorem{example}[proposition]{Example}
\newtheorem{remark}[proposition]{Remark}
\numberwithin{equation}{section}
\title[Reduction of integration domain]{Reduction of integration domain in Triebel--Lizorkin spaces}
\author{Artur Rutkowski}
\address{Faculty of Pure and Applied Mathematics,
	Wroc\l aw University of Science and Technology,
	Wyb. Wyspia\'nskiego 27, 50-370 Wroc\l aw, Poland}
\email{artur.rutkowski@pwr.edu.pl}
\thanks{\hspace{-11pt}2010 \textit{Mathematics Subject Classification}. Primary 46E35 Secondary 31C25.}

\begin{document}
	\begin{abstract}
	We investigate the comparability of generalized Triebel--Lizorkin and Sobolev seminorms on uniform and non-uniform sets when the integration domain is truncated according to the distance from the boundary. We provide numerous examples of kernels and domains in which the comparability does and does not hold.
\end{abstract}
\maketitle
\section{Introduction}

Let $\Omega\subset \mR^d$ be a domain, $d\geq 1$, and let $p,q\in (1,\infty)$. Let $K\colon \mR^d \times \mR^d \to (0,\infty]$ be a homogeneous, radial kernel, i.e. $K(x,y) = k(|x-y|)$, satisfying $\int_{\mR^d}  (1\wedge|y|^q)K(0,y) \, dy <\infty$. We define the (generalized) Triebel--Lizorkin space on $\Omega$ as
\begin{equation}\label{eq:tldef}
\fpqom := \Big\{f\in L^p(\Omega): \intom\bigg(\intom |f(x) - f(y)|^q K(x,y) \,\red{dy}\bigg)^{\frac pq} \, \red{dx} < \infty \Big\}.
\end{equation}
The space $\fpqom$ obviously depends on $K$, however we skip it in the notation for clarity.

$\fpqom$ is endowed with the norm
$$\| f\|_{\fpqom} = \|f\|_{L^p(\Omega)} + \bigg(\intom\Big(\intom |f(x) - f(y)|^q K(x,y) \, \red{dy}\Big)^{\frac pq} \, \red{dx} \bigg)^{\frac 1p}.$$

We are interested in the Gagliardo-type seminorm
\begin{equation}\label{eq:fullsemi}
\bigg(\intom\Big(\intom |f(x) - f(y)|^q K(x,y) \, \red{dy}\Big)^{\frac pq} \, \red{dx} \bigg)^{\frac 1p},
\end{equation}
which will be called the \textit{full seminorm}. Let \red{$\theta \in (0,1]$} and let $\delta(x) = d(x,\partial \Omega)$. Our main goal is to establish the comparability of the full seminorm and the \textit{truncated seminorm}
\begin{equation}\label{eq:truncated}
\bigg(\intom\Big(\int_{B(\red{x},\theta\delta(\red x))} |f(x) - f(y)|^q K(x,y) \, \red {dy}\Big)^{\frac pq} \, \red {dx} \bigg)^{\frac 1p}
\end{equation}
for sufficiently regular $K$ and $\Omega$. Later on, such occurrence will be called a \textit{comparability result}.

Here is our first comparability result.
\begin{theorem}\label{th:main}
	Assume that $\Omega$ is a uniform domain, and that $K$ satisfies \AJ, \AD\ and \AT\ \red{formulated in Subsection \ref{sec:a} below}. Assume that $1<q\leq p<\infty$. Then for every $0<\theta\leq 1$
	\begin{align*}\bigg(\intom\Big(\intom |f(x) - f(y)|^q K(x,y) \, \red {dy}\Big)^{\frac pq} \, \red {dx} \bigg)^{\frac 1p} \approx \bigg(\intom\Big(\int_{B(\red x,\theta\delta(\red x))} |f(x) - f(y)|^q K(x,y) \, \red{dy}\Big)^{\frac pq} \, \red{dx} \bigg)^{\frac 1p}.\end{align*}
	The comparability constant depends on $p,q,\theta,\Omega$, and the constants in assumptions \AD, \AT.
\end{theorem}

This is a generalization of the result of Prats and Saksman \cite[Theorem 1.6]{PS} who prove it for the kernels of the form: $K(x,y) = |x-y|^{-d-qs}$ for $s\in (0,1)$. Recently, we were informed that this classical case can also be resolved using the much earlier result of Seeger \cite[Corollary 2]{MR1097208}. 

Another result of this flavor was established by Dyda \cite[(13)]{DYDA2006564} and was used to obtain Hardy inequalities for nonlocal operators. \red{More recent results on reduction of integration domain in fractional Sobolev spaces include Bux, Kassmann, and Schulze \cite{2017arXiv170709277B} who consider certain cones with apex at $x$ instead of $B(x,\delta(x))$, and Chaker and Silvestre \cite{2019arXiv190413014C}.} 

Here we mention that, independently of our work, Kassmann and Wagner \cite{2018arXiv181012289K} have also proved comparability results which extend the ones from \cite{PS}, allowing for kernels with scaling conditions \red{for $p=q=2$}. However, their overall aim and scope are different than ours.

In Theorem \ref{th:main} we adapt the method of proof from \cite{PS} for a wide class of kernels of the form $K(x,y) = |x-y|^{-d}\phi(|x-y|)^{-q}$. The most technical assumption \textbf{A2} is tailored for the key Lemmas \ref{lemMaximal}, \red{\ref{lem:phil}}, however in Subsection \ref{sec:zal} we argue that it amounts to \red{at least a power-type decay at 0, and for unbounded $\Omega$ at least a power-type growth at $\infty$, of $\phi$.} For the formulation of the assumptions, see Subsection \ref{sec:a}.

Notably, we go beyond the uniform domains, where the methods used by Prats and Saksman are no longer available. Namely, we prove that the comparability may hold for the fractional Sobolev spaces in strip domains. 

\begin{theorem}\label{th:stripmd}
	Assume that $p=q=2$. Let $\Omega = \mR^k\times (0,1)^l \subseteq \mR^{k+l}$ with $k,l>0$. For $d=k+l$ let $K(x,y) = |x-y|^{-d-\alpha}$ with $\alpha\in (0,2)$. If \red{$k-l - \alpha < -1$}, then the seminorms \eqref{eq:fullsemi} and \eqref{eq:truncated} are comparable.
\end{theorem}

We also construct a counterexample for $\alpha < 1$ and $k=l=1$. This shows an intriguing interplay between the kernel and the width of the domain. \red{Heuristically speaking, it can be seen both in Theorem \ref{th:stripmd} and in Subsection \ref{sec:zal} that the comparability holds if the stochastic process corresponding to the jump kernel $K\cdot \textbf{1}_{\Omega\times \Omega}$ and the shape of the domain $\Omega$ favor small jumps over large jumps. We remark that the connection between the jump kernel and the stochastic process is a very delicate matter. In Section \ref{sec:proces} we present a short discussion of this subject and we place our comparability results in this context}.

Another object of our studies is the 0-order kernel $K(x,y) \approx |x-y|^{-d}$. We provide examples showing that the comparability does not hold in this case. In an attempt to repeat the proof of Theorem \ref{th:main} we obtain an estimate of \eqref{eq:fullsemi} by a truncated seminorm with a slightly more singular kernel, \red{see Theorem \ref{th:0order} below.}


The classical Triebel--Lizorkin spaces were introduced independently by Lizorkin \cite{Liz74} and Triebel \cite{triebel1973}. The original definition is formulated using Paley--Littlewood theory and is widely used in analysis and applications, see e.g. \cite{doi:10.1002/mana.201600315,bu2009,grafakos2014modern}. For various cases of $p,q,d$ and $\Omega$ the classical definition was proved to be equivalent to $\eqref{eq:tldef}$ with $K(x,y) = |x-y|^{-d-sq}$, where $s\in(0,1)$, see \cite{PS,stein1961,triebel2010theory}.

The definition \eqref{eq:tldef} seems more natural if the starting point is $p=q=2$, e.g., fractional Sobolev spaces in nonlocal PDEs \cite{DINEZZA2012521, MR3318251, 2018AR}, or Dirichlet forms for Hunt processes \cite{MR2006232, MR2778606}. It is also a suitable definition for considering kernels $K$ more general than $|x-y|^{-d-sq}$ which is also of interest in the field of nonlocal operators and stochastic processes. In this paper we will not attempt to characterize the definition \eqref{eq:tldef} in the spirit of classical definitions by Triebel and Lizorkin in the full generality. However, we \red{use} the Fourier methods in Section 5, where we compare spaces with kernels which are only slightly different from each other.

As we argue further in the article, the comparability results can be used to study a class of stochastic processes, whose jumps from \red{$x$} are restricted to the ball $B(\red x,\theta \delta(\red x))$. The truncated seminorms may also prove useful in peridynamics, as $B(\red x,\theta\delta(\red x))$ may be understood as the \textit{variable horizon}, see e.g. \cite{varhor2,varhor}, and in particular Du and Tian \cite{doi:10.1137/16M1078811} where horizons depending on the distance from the boundary are studied.

The paper is organized as follows. Section 2 contains notions, assumptions, and basic facts used further in our work. Section 3 is devoted to proving Theorem \ref{th:main}. In Section 4 we present positive and negative examples of kernels concerning the comparability results. Section 5 contains the analysis of 0-order kernels. In Section 6 we consider strip domains, in particular we prove Theorem \ref{th:stripmd}. Section 7 presents the connection of our development with the theory of Hunt processes.
\section*{Acknowledgements}
I am grateful to Bart\l omiej Dyda for introduction to the subject, many hours of helpful discussions and for reading the manuscript. I thank Tomasz Grzywny and Dariusz Kosz for stimulating discussions. I also thank Mart\'i Prats for pointing out a flaw in the proof of Theorem \ref{th:main} in the previous version of the manuscript.
\red{I express my gratitude to the anonymous referees for numerous essential remarks and for raising important questions concerning the proofs and the main assumptions of the paper. Research was supported by the grant 2015/18/E/ST1/00239 of the National Science Center (Poland).}
\section{Preliminaries and assumptions}
\subsection{Assumptions on the kernel}\label{sec:a}
\red{In the sequel we will consider the exponents $1<q\leq p<\infty$ and the assumptions are subordinated to them. As usual, $p' = \frac p{p-1}$ is the H\"older conjugate of $p$. We also let
	$$N(r) = \inf\{k\in \mathbb{N}: 2^kr > \diam (\Omega)\},\quad r>0.$$
	We have $N(r) = \infty$ for every $r>0$ if and only if $\Omega$ is unbounded.}\\
We assume that the kernel $K$ is of the form $K(x,y) = |x-y|^{-d}\phi(|x-y|)^{-q}$, where $\phi\colon (0,\infty) \to (0,\infty)$ satisfies
\begin{enumerate}
	\item[\AJ] $(1\wedge|y|^q)|y|^{-d}\phi(|y|)^{-q}\in L^1(\mR^d)$,
	\item[\AD] $\phi$ is increasing \red{and there exists $C_2>0$ such that for $t_1 = \min(q,p - \frac pq)$, $t_2= \frac 1{q-1}$, and for every $0<r<\diam(\Omega)$, we have $$\sum\limits_{k=1}^{N(r)} \frac{\phi(r)^{t_1}}{\phi(2^kr)^{t_1}} \leq C_2$$ and $$\sum\limits_{k=1}^\infty \frac{\phi(2^{-k}r)^{t_2}}{\phi(r)^{t_2}} \leq C_2.$$}
	\item[\AT] \red{There exists $C_3\geq 1$ such that for every $0<r<3\diam(\Omega)$, we have $\phi(2r) \leq C_3 \phi (r)$.}
\end{enumerate}
In particular, we allow unbounded domains in which the scaling conditions \AD,\ \AT\ become global. Note that \AJ\ is a L\'evy \red{measure}-type condition\red{, which assures the finiteness of \eqref{eq:fullsemi} for smooth, compactly supported~$u$}. If $q=2$ and $\phi(r) = r^{s}$, $s\in (0,1)$, then $K$ corresponds to the fractional Laplacian of order $s$ and all the assumptions are satisfied. The conditions \AD \ and \AT\ imply certain \red{scaling} for $K$, see Subsection \ref{sec:zal} for the details. \red{The exponents $t_1$ and $t_2$ in \AD\ stem from the five instances of usage of Lemma \ref{lemMaximal} and \ref{lem:phil} in the proof of Theorem \ref{th:main}. Since $\phi$ is increasing, the bounds in \AD\ hold for all larger exponents in place of $t_1$ and $t_2$. We note the following, frequently used below, consequence of \AT\ and the monotonicity of $\phi$: if $x \lesssim y$, then $\phi(y)^{-1} \lesssim \phi(x)^{-1}$.}
\subsection{Whitney decomposition and uniform domains}
For cubes $Q, R$ in $\mR^d$ we consider $l(R)$ --- the length of the side of $R$, and the long distance between $Q$ and $R$: $D(Q,R) = l(Q) + d(Q,R) + l(R)$, where $d$ is the Euclidean distance. The scaling of the cube is done from its center $x_Q$.

We say that a family of (closed) dyadic cubes $\mathcal{W}$ is a Whitney decomposition of $\Omega$ if for every $Q, S \in \mW$
\begin{itemize}
	\item if $Q\neq S$, then  $\intr(Q)\cap\intr(S) = \emptyset;$
	\item if $Q\cap S \neq \emptyset$, then $l(Q)\leq 2 l(S)$;
	\item if $Q\subseteq 5S$, then $l(S) \leq 2l(Q)$;
	\item there is a constant $C_{\mW}$ such that $C_{\mW}l(Q) \leq d(Q,\partial\Omega)\leq 4C_{\mW}l(Q)$.
\end{itemize}
A sequence of cubes $(Q,R_1,\ldots,R_n,S)$ is a chain connecting $Q$ and $S$, if every cube \red{is a neighbor of} its successor and predecessor (if it has one), \red{by which we mean that their boundaries have non-empty intersection}. We will denote the chain as $[Q,S]$ and the sum of the lengths of its cubes as $l([Q,S])$. We let $[Q,S)= [Q,S]\setminus \{S\}$.

The Whitney decomposition is admissible, if there exists $\eps > 0$ such that for every pair of cubes $Q,S$, there exists an $\eps$-admissible chain $[Q,S] = (Q_1,Q_2,\ldots Q_n)$, i.e.
\begin{itemize}
	\item $l([Q,S]) \leq \frac 1\eps D(Q,S)$,
	\item there exists $j_0 \in \{1,\ldots, n\}$ for which $l(Q_{j}) \geq \eps D(Q,Q_j)$ for every $1\leq j \leq j_0$, and $l(Q_{j}) \geq \eps D(Q_j,S)$ for every $j_0\leq j\leq n$. $Q_{j_0}$ will be denoted as $Q_S$ --- the central cube of the chain $[Q,S]$.
\end{itemize}
A domain which has an admissible Whitney decomposition is called a uniform domain. Unless we state otherwise, $[Q,S]$ is an arbitrary ($\eps$-)admissible chain connecting $Q$ and $S$. 

The shadow of a cube is $\Sh_\rho(Q) = \{S\in \mW: S\subseteq B(x_Q,\rho l(Q))\}$, $\rho > 0$. We also denote $\SH_\rho(Q) = \bigcup \Sh_\rho(Q)$. Note that we can take a sufficiently large $\rho_\eps$ so that
\begin{itemize}
	\item for every $\eps$-admissible chain $[Q,S]$, and every $P\in [Q,Q_S]$, we have $Q \in \Sh_{\rho_\eps}(P)$,
	\item if $[Q,S]$ is $\eps$-admissible, then every cube from it belongs to $\Sh_{\rho_\eps}(Q_S)$,
	\item for every $Q\in \mW$, $5Q \subseteq \SH_{\rho_\eps}(Q)$.
\end{itemize}
From now on we fix $\rho_\eps$ and write $\Sh(Q) = \Sh_{\rho_\eps}(Q)$ and $\SH(Q) = \SH_{\rho_\eps}(Q)$.
\begin{remark}
	\red{The proofs appearing throughout the paper involve a lot of '$\lesssim$', and '$\gtrsim$' signs. We would like to stress that any comparability for $\phi$ stems from \AD\ and \AT. In particular, for fixed $p,q$ the constants can be chosen to depend only on the geometry of $\Omega$ (including the dimension) and the constants in \AD\ and \AT\ wherever $\phi$ is used.}
\end{remark}

The next lemma provides some inequalities for the \red{non-centered} Hardy--Littlewood maximal operator (denoted by $M$) with connection to the kernel $K$. \red{It is inspired by the results of \cite[Section 2]{PS} and Prats and Tolsa \cite[Section 3]{PRATS20152946}.}
\begin{lemma}\label{lemMaximal}
	Let $\Omega$ be a domain with Whitney covering $\mW$, and let $\phi$ satisfy \AJ ,\ \AD\ and \AT. Assume that $g\in L^1_{loc}(\mR^d)$ \red{is nonnegative} and $0<r<3\diam(\Omega)$. For every \red{$\eta \geq \min(q,p - \frac pq)$}, $Q\in\mW$ and $x\in \red{\Omega}$, we have
	\begin{align}\label{eqMaximalFar}
	\int_{\red{\Omega \cap \{|x-y|>r\}}} \frac{g(y) \, dy}{|x-y|^d\phi(|x-y|)^\eta}&\lesssim \frac{Mg(x)}{\phi(r) ^\eta},\\
	\sum_{S:D(Q,S)>r}  \frac{\int_S g(y) \, dy}{D(Q,S)^d\phi(D(Q,S))^\eta}&\lesssim \frac{\inf_{x\in Q} Mg(x)}{\phi(r) ^\eta},\label{eqMaximalfar2}
	\end{align}
	and
	\begin{equation}\label{eqMaximalAllOver}
	\sum_{S\in\mathcal{W}} \frac{l(S)^d}{D(Q,S)^d\phi(D(Q,S))^{\eta}} \lesssim \frac{1}{\phi(l(Q))^\eta}.
	\end{equation}
\end{lemma}
\begin{proof}
	Let us look at \eqref{eqMaximalFar}. For clarity, assume that $\Omega \ni x=0$. Since $1/\phi$ is decreasing, we get
	\begin{align*}
	\int_{\red{\Omega \cap \{|y|>r\}}} \frac{\phi(r)^\eta g(y) \, dy}{|y|^d\phi(|y|)^\eta} &\red{\leq} \sum\limits_{k=1}^{\red{N(r)}} \int_{2^{k-1}r<|y|<2^kr} \frac {g(y)}{|y|^d}\frac {\phi(r)^\eta}{\phi(|y|)^\eta} \, dy\\
	&\lesssim \sum\limits_{k=1}^{\red{N(r)}} \frac {\phi(r)^\eta}{\phi(2^{k-1}r)^\eta}\frac 1 {|B_{2^kr}|}\int_{2^{k-1}r<|y|<2^kr} g(y) \, dy \\
	&\leq \sum\limits_{k=1}^{\red{N(r)}} \frac {\phi(r)^\eta}{\phi(2^{k-1}r)^\eta} Mg(0).
	\end{align*}
	The sum is bounded with respect to $r$ thanks to \AD.  
	In order to prove \eqref{eqMaximalfar2} note that if $D(Q,S)>r$, then for every $x\in Q$, $y\in S$, we have $|x-y| + r \lesssim D(Q,S)$. Therefore, by \AT\ \red{and the fact that $\phi$ is increasing,}
	for every $x\in Q$ we have
	\begin{align*}
	&\sum_{S:D(Q,S)>r} \frac {\phi(r)^\eta \int_S g(y) \, dy}{D(Q,S)^d\phi(D(Q,S))^\eta} \lesssim \int_{\red{\Omega}} \frac{\red{\phi(r)^\eta g(y)} \, dy}{(|x-y|+ r)^d\phi(|x-y| + r)^\eta}\\
	&\leq \int_{\red{\Omega \cap\{|x-y| > r\}}} \frac{\phi(r)^\eta g(y) \, dy}{|x-y|^d\phi(|x-y|)^\eta} +  \int_{|x-y|<r} \frac {\phi(r)^\eta g(y) \, dy }{\red{r^d\phi(r)^\eta}}\\
	&\red{\lesssim} \int_{\red{\Omega \cap \{|x-y| > r\}}} \frac{\phi(r)^\eta g(y) \, dy}{|x-y|^d\phi(|x-y|)^\eta} + \frac 1 {|B_r|}\int_{|x-y|<r}g(y)\, dy.
	\end{align*}
	The claim follows from the previous estimate. Since the constants in the inequalities \red{do} not depend on $x$, the same holds for the infimum.
	
	\red{Inequality }\eqref{eqMaximalAllOver} can be obtained by taking $g\equiv 1$ and $r = l(Q)$ in \eqref{eqMaximalfar2}. In that case $D(Q,S) > r$ for every $S$, including $Q$.
\end{proof}
\red{The following lemma is an extension of \cite[(2.7),(2.8)]{PS}.}
\begin{lemma}\label{lem:phil}
	\red{Let $\eta \geq \min(q,p-\frac pq)$, $\kappa\geq \frac 1{q-1}$, assume that \AD\ and \AT\ hold, and assume that $\mW$ is admissible. Then}
	\begin{equation}\label{eq:insh}
	\sum_{R: P \in \Sh_\rho(R)} \phi(l(R))^{-\eta} \lesssim \phi(l(P))^{-\eta}.
	\end{equation}
	Furthermore, if $S\in\Sh_{\red{\rho}}(R)$, then 
	\begin{equation}\label{eq:chain}
	\sum_{P\in [S,R]}\phi(l(P))^{\kappa} \lesssim \phi(l(R))^{\kappa}.
	\end{equation}
\end{lemma}
\begin{proof}
	\red{Since the cubes are dyadic, we may and do assume in \eqref{eq:insh} that $l(P) = 2^{p_0}$ for some $p_0\in \mathbb{Z}$. Every $R$ which satisfies $P\in \Sh_\rho(R)$ must be at a distance from $P$ smaller than a multiple of $l(R)$, therefore there can only be a bounded number $K$ of such cubes $R$ with a given side length. Furthermore, the considered cubes must be sufficiently large to contain $P$ in its shadow, that is $l(R) \geq 2^{p_0-l_0}$ with $l_0\in \mathbb{N}_0$ independent of $p_0$. We also obviously have $l(R) < \diam (\Omega)$. Thus, the sum in the first assertion can be bounded from above as follows:
		\begin{align*}
		\sum_{R: P \in \Sh_\rho(R)} \phi(l(R))^{-\eta} \leq K\sum_{k=p_0-l_0}^{p_0+N(2^{p_0})} \phi(2^k)^{-\eta}= K\sum_{k=p_0-l_0}^{p_0} \phi(2^k)^{-\eta}+ K\sum_{k=p_0+1}^{p_0+N(2^{p_0})} \phi(2^k)^{-\eta}.
		\end{align*}
		The sums are estimated by a multiple of $\phi(2^{p_0})^{-\eta}$ using \AT\ and \AD\ respectively, which proves \eqref{eq:insh}.}
	
	\red{As in the proof of \cite[(2.8)]{PS} we may deduce that if $S\in \Sh_\rho(R)$, then there is a bounded number $L$ of cubes $P\in [S,R]$ of a given side length. Furthermore, for every $P\in[S,R]$ we have $l(P)\leq 2^{r_0+l_0}$, where $l(R) = 2^{r_0}$ and $l_0$ is a fixed natural number independent of $S$ and $R$. Therefore we estimate \eqref{eq:chain} as follows:
		\begin{align*}
		\sum_{P\in [S,R]}\phi(l(P))^{\kappa} \leq L\sum_{k=-\infty}^{r_0+l_0} \phi(2^k)^\kappa = L\sum_{k=-\infty}^{r_0}\phi(2^k)^\kappa + L\sum_{k=r_0+1}^{r_0+l_0} \phi(2^k)^\kappa.
		\end{align*}
		The first sum is bounded from above by a multiple of $\phi(2^{r_0})^\kappa$ because of the second assertion of \AD\ and the second is handled by using \AT. This ends the proof.}
\end{proof}
\section{Proof of Theorem \ref{th:main}}
\begin{proof}[Proof of Theorem \ref{th:main}]
	Obviously it suffices to show that the truncated seminorm dominates the full one up to a multiplicative constant.
	
	We will work with dual norms, namely
	\begin{equation}\label{eq:dual}
	\sup\limits_{\substack{g\geq 0\\ \|g\|_{L^{p'}(L^{q'}(\Omega))}\leq 1}} \intom\intom |f(x) - f(y)| |x-y|^{-\frac dq} \phi(|x-y|)^{-1} g(x,y) \, dy \, dx.
	\end{equation}
	From now on, $g$ will be like in formula \eqref{eq:dual}.
	
	First let us take care of the case when $x$ and $y$ are close to each other. By the H\"older's inequality, we get
	\begin{align*}
	&\sum_{Q\in\mW}\int_Q\int_{2Q} \frac{|f(x)-f(y)|g(x,y)}{|x-y|^{\frac dq}\phi(|x-y|)}\, dy \, dx\\ \leq &\sum_{Q\in\mW}\int_Q\Big(\int_{2Q} \frac{|f(x) - f(y)|^q}{|x-y|^d\phi(|x-y|)^q} \, dy\Big)^{\frac 1q} \Big(\int_{2Q} g(x,y)^{q'} \, dy\Big)^{\frac 1{q'}} \, dx\\
	\leq &\bigg(\sum_{Q\in\mW} \int_Q\Big(\int_{2Q} \frac{|f(x) - f(y)|^q}{|x-y|^d \phi(|x-y|)^q} \, dy\Big)^{\frac pq}\, dx\bigg)^{\frac 1p}.
	\end{align*}
	
	What remains is the integral over $(\Omega\times\Omega) \setminus \bigcup_{Q\in\mW} Q\times 2Q = \bigcup_{Q\in\mW}Q\times(\Omega\setminus 2Q) = \bigcup_{Q,S\in\mW}Q\times(S\setminus 2Q)$. We claim that in this case $|x-y| \approx D(Q,S)$. Indeed, since $y \notin 2Q$, we immediately get $l(Q) \red{\leq} |x-y|$. Furthermore, if $l(S)\geq l(Q)$, and $|x-y|\leq 2l(S)$, then $Q\subseteq 5S$, and by the definition of the Whitney decomposition $l(Q) \geq \frac 12 l(S)$ which proves the claim. Therefore, by \AT\
	we get
	\begin{align}\nonumber
	&\sum_{Q,S}\int_Q\int_{S \setminus 2Q} \frac{|f(x) - f(y)|g(x,y)}{|x-y|^{\frac dq}\phi(|x-y|)} \, dy \, dx\\
	 \lesssim &\sum_{Q,S}\int_Q\int_S \frac{|f(x) - f(y)|g(x,y)}{D(Q,S)^{\frac dq}\phi(D(Q,S))} \, dy \, dx.\label{eq:podd}
	\end{align}
	Let $f_Q = \frac 1{|Q|} \int_Q f(x) \, dx$. By the triangle inequality \eqref{eq:podd} does not exceed
	\begin{align}
	&\sum_{Q,S}\int_Q\int_S \frac{|f(x) - f_Q|g(x,y)}{D(Q,S)^{\frac dq}\phi(D(Q,S))}\, dy \, dx\tag{\textbf{A}}\label{A}\\
	+&\sum_{Q,S}\int_Q\int_S \frac{|f_Q - f_{Q_S}|g(x,y)}{D(Q,S)^{\frac dq}\phi(D(Q,S))}\, dy \, dx.\tag{\textbf{B}}\label{B}\\
	+&\sum_{Q,S}\int_Q\int_S \frac{|f_{Q_S} - f(y)|g(x,y)}{D(Q,S)^{\frac dq}\phi(D(Q,S))}\, dy \, dx.\tag{\textbf{C}}\label{C}
	\end{align}
	Using H\"older's inequality and \eqref{eqMaximalAllOver} we can estimate \eqref{A} from above by
	\begin{align}
	&\sum_Q \int_Q |f(x) - f_Q| \Big(\intom g(x,y)^{q'} \, dy\Big)^\frac 1{q'}\Big(\sum_S \frac {l(S)^d}{D(Q,S)^d\phi(D(Q,S))^q}\Big)^{\frac 1q} \, dx\nonumber\\
	\lesssim &\sum_Q \int_Q |f(x) - f_Q|\Big(\intom g(x,y)^{q'} \, dy\Big)^{\frac 1{q'}}\frac 1{\phi(l(Q))} \, dx\label{eq:repA}\\
	\lesssim &\bigg(\sum_Q \int_Q \Big(\frac{|f(x) - f_Q|}{\phi(l(Q))}\Big)^p \, dx\bigg)^{\frac 1p}.\nonumber
	\end{align}
	Now, by the definition of $f_Q$, Jensen's inequality, and \red{\AT}\ we get
	\begin{align*}
	\AAA\, &\red{\lesssim} \bigg(\sum_Q \int_Q \Big(\int_Q\frac{|f(x) - f(y)|^q}{l(Q)^d\phi(l(Q))^q}\, dy\Big)^{\frac pq} \, dx\bigg)^{\frac 1p}\\ &\red{\lesssim} \bigg(\sum_Q \int_Q \Big(\int_Q\frac{|f(x) - f(y)|^q}{|x-y|^d\phi(|x-y|)^q}\, dy\Big)^{\frac pq} \, dx\bigg)^{\frac 1p}.
	\end{align*}
	Let us consider \eqref{B}. If we denote the successor of \red{$P$} in a chain $[Q,S\red{)}$ as \red{$\mN(P)$}, then by the triangle inequality
	\begin{align*}
	\BB \leq \sum_{Q,S}\Bigg(\int_Q\int_S \frac{g(x,y)}{D(Q,S)^{\frac dq}\phi(D(Q,S))}\, dy \, dx \sum_{P\in[Q,Q_S)} |f_P - f_{\mN(P)}|\Bigg).
	\end{align*}
	Recall that $\mN(P)\subseteq 5P$, and for every $P\in [Q,Q_S]$, $Q\in \Sh(P)$. For such $P$ it is also true that $D(P,S) \approx D(Q,S)$, see \cite[(2.6)]{PS}. Therefore, by \AT\ we estimate \BB\ from above by a multiple of
	\begin{align*}
	\red{\sum_P\int_P\int_{5P} \frac{|f(\xi) - f(\zeta)|}{|P||5P|}\, d\xi \, d\zeta}\sum_{Q\in \Sh(P)}\int_Q\sum_S\int_S\frac {g(x,y)}{D(P,S)^{\frac dq}\phi(D(P,S))} \, dy \, dx.
	\end{align*}
	By the H\"older's inequality and \eqref{eqMaximalAllOver} this expression approximately less than or equal to
	\begin{align}
	\sum_P\int_P\int_{5P} \frac{|f(\xi) - f(\zeta)|}{|P||5P|}\, d\xi \, d\zeta \int_{\SH(P)}\Big(\int_\Omega g(x,y)^{q'} \, dy\Big)^{\frac 1{q'}}\frac 1{\phi(l(P))} \, dx.\label{eq:repB} 
	\end{align}
	Let $G(x) = \big(\int_\Omega g(x,y)^{q'} \, dy\big)^{\frac 1{q'}}$. By \cite[Lemma 2.7]{PS} we have $\int_{\SH (P)} G(x) \, dx \hspace{-2.07pt}\lesssim \inf\limits_{y\in P} MG(y) l(P)^d$. Using this, the Jensen's inequality, the H\"older's inequality, and the fact that the maximal operator is continuous in $L^{p'}$, $p'>1$, we obtain
	\begin{align*}
	\BB &\lesssim \sum_P\frac 1{|P||5P|}\frac {l(P)^d}{\phi(l(P))}\int_P\int_{5P} |f(\xi) - f(\zeta)|MG(\zeta)\, d\xi \, d\zeta\nonumber\\
	&\lesssim \sum_P\int_P\frac {MG(\zeta)}{l(P)^{\frac dq}\phi(l(P))}\bigg(\int_{5P} |f(\xi) - f(\zeta)|^q \, d\xi\bigg)^{\frac 1q} \, d\zeta\nonumber\\
	&\lesssim \Bigg(\sum_P\int_P\bigg(\int_{5P}\frac{|f(\xi) - f(\zeta)|^q}{l(P)^d\phi(l(P))^q} \, d\xi\bigg)^{\frac pq} \, d\zeta\Bigg)^{\frac 1p}.
	\end{align*}
	Since $|\xi - \zeta| \leq 5l(P)$, \BB\ is estimated.
	
	Now \red{we will work on \eqref{C}. Since $D(Q,S) \approx l(Q_S)$, by \AT\ we obtain
		\begin{equation*}
		\CC \lesssim \sum\limits_{Q,S}\int_Q\int_S \frac{|f_{Q_S} - f(y)|g(x,y)}{l(Q_S)^{\frac dq}\phi(l(Q_S))}\, dy\, dx.
		\end{equation*}
		Furthermore, for every admissible chain we have $Q,S\in \Sh(Q_S)$, therefore for every $Q,S \in \mW$ we have
		\begin{equation*}
		(Q_S,Q,S)\in \bigcup\limits_{R\in \mW}\{(R,P,P') :  P,P' \in \Sh (R)\}.
		\end{equation*}
		Consequently, the following estimate holds:}
	\begin{align}
	\label{Cstart}\CC&\lesssim \sum\limits_{R\in \mW}\sum\limits_{Q\in\Sh(R)}\sum\limits_{S\in\Sh(R)} \int_Q\int_S \frac{|f_{R} - f(y)|g(x,y)}{l(R)^{\frac dq} \phi(l(R))}\, dy \, dx.
	\end{align}
	By H\"older's inequality the above expression does not exceed
	\begin{align*}
	&\sum_{R\in\mW}\frac{\Big(\int_{\SH(R)} |f_R - f(y)|^q \, dy\Big)^{\frac 1q}}{l(R)^{\frac dq}\phi(l(R))}\int_{\SH(R)}\Bigg(\int_{\SH(R)} g(x,y)^{q'} \, dy\Bigg)^{\frac {1}{q'}} \, dx\\
	\leq &\sum_{R\in\mW}\frac{\Big(\int_{\SH(R)} |f_R - f(y)|^q \, dy\Big)^{\frac 1q}}{l(R)^{\frac dq}\phi(l(R))}\int_{\SH(R)} G(x) \, dx.
	\end{align*}
	\red{By the last estimate of \cite[Lemma 2.7]{PS}}, the fact that $\inf\limits_R MG\leq \frac 1{l(R)^d}\int_R MG$, and the H\"older's inequality we get that
	\begin{align*}
	\CC &\lesssim \sum_{R\in\mW}\frac{1}{l(R)^{\frac {d}q}\phi(l(R))}\Bigg(\int_{\SH(R)} |f_R - f(y)|^q \, dy\Bigg)^{\frac 1q} \int_R MG(\xi) \, d\xi\\
	&\leq \Bigg(\sum_{R\in\mW}\int_R\frac{1}{l(R)^{\frac {dp}q}\phi(l(R))^p}\Big(\int_{\SH(R)} |f_R - f(y)|^q \, dy \Big)^{\frac pq}\, d\xi \Bigg)^{\frac 1p}\hspace{-3pt} \|MG\|_{L^{p'}(\Omega)}\\
	&\leq \Bigg(\sum_{R\in\mW}\frac{l(R)^d}{l(R)^{\frac {dp}q}\phi(l(R))^p}\Big(\sum_{S\in\Sh(R)}\int_{S} |f_R - f(y)|^q \, dy \Big)^{\frac pq}\Bigg)^{\frac 1p}.
	\end{align*}
	Let $[S,R]$ be an admissible chain between $S$ and $R$. Then, after using the inequality $|f_R - f(y)|^q \lesssim |f_R - f_S|^q + |f_S - f(y)|^q$, we get
	\begin{align*}
	\CC^p &\lesssim \sum_{R\in\mW} \frac{l(R)^d}{l(R)^{\frac {dp}q}\phi(l(R))^p}\Bigg(\sum_{S\in\Sh(R)}\Big|\sum_{P\in [S,R)}f_P - f_{\mN(P)}\Big|^{q} l(S)^d\Bigg)^{\frac pq}\\
	&+ \sum_{R\in\mW} \frac{l(R)^d}{l(R)^{\frac {dp}q}\phi(l(R))^p}\Bigg(\sum_{S\in\Sh(R)}\int_S |f_S - f(y)|^q \, dy\Bigg)^{\frac pq} = \CJ + \CD.
	\end{align*}
	If we write $f_P - f_{\mN(P)} = (f_P - f_{\mN(P)}) \frac{\phi(l(P))^{\frac 1q}}{\phi(l(P))^{\frac 1q}}$, then by H\"older's inequality we estimate \CJ\ from above by
	\begin{align*}
	&\sum_{R\in\mW} \frac{l(R)^d}{l(R)^{\frac {dp}q}\phi(l(R))^p}\Bigg(\sum_{S\in\Sh(R)}\sum_{P\in \red{[S,R)}}\frac{|f_P - f_{\mN(P)}|^ql(S)^d}{\phi(l(P))} \Big(\sum_{P\in \red{[S,R)}}\phi(l(P))^{\frac {q'}{q}}\Big)^{\frac {q}{q'}}\Bigg)^{\frac pq}.
	\end{align*}
	By Lemma \ref{lem:phil}
	\begin{equation*}
	\CJ \lesssim \sum_{R\in\mW} \frac{l(R)^d}{l(R)^{\frac {dp}q}}\phi(l(R))^{\frac pq -p}\bigg(\sum_{S\in\Sh(R)}\sum_{P\in \red{[S,R)}}\frac{|f_P - f_{\mN(P)}|^q}{\phi(l(P))} l(S)^d\bigg)^{\frac pq}.
	\end{equation*}
	Let us take $\rho_2$ large enough for $S\in\Sh^2(\red{P}) := \Sh_{\rho_2}(\red{P})$, and $P\in \Sh^2(R)$ to hold. Then $\sum_{S\in\Sh(R)}\sum_{P\in \red{[S,R)}} \lesssim \sum_{P\in \Sh^2(R)}\sum_{S\in\Sh^2(P)}$. We denote the sum of the neighbors of $P$ as $U_P$. Since $\sum_{S\in \Sh^2(P)} l(S)^d\lesssim l(P)^d$, we get that, up to a multiplicative constant, \CJ\ does not exceed
	\begin{equation*}
	\sum_{R\in\mW} \frac{l(R)^d}{l(R)^{\frac {dp}q}}\phi(l(R))^{\frac pq -p}\Bigg(\sum_{P\in\Sh^2(R)}\frac{\red{(l(P)^{-d}\int_{U_P} |f_P - f(\xi)|\, d\xi)}^q}{\phi(l(P))} l(P)^d\Bigg)^{\frac pq}.
	\end{equation*}
	Since $p\geq q$, we can use the H\"{o}lder's inequality with exponent $\frac pq$ to estimate this expression from above by
	\begin{align*}
	&\sum_{R\in\mW} \frac{l(R)^d}{l(R)^{\frac {dp}q}}\phi(l(R))^{\frac pq -p}\Bigg(\sum_{P\in\Sh^2(R)}\frac{\red{(l(P)^{-d}\int_{U_P} |f_P - f(\xi)|\, d\xi)}^p}{\phi(l(P))^{\frac pq}} l(P)^d\Bigg)\Bigg(\sum_{P\in\Sh^2(R)} l(P)^d\Bigg)^{(1 - \frac qp)\frac pq}\\
	\lesssim &\sum_{R\in\mW}\sum_{P\in\Sh^2(R)} \phi(l(R))^{\frac pq -p}\frac{\red{(l(P)^{-d}\int_{U_P} |f_P - f(\xi)|\, d\xi)}^pl(P)^d}{\phi(l(P))^\frac pq}\\
	\lesssim &\sum_{P\in \mW}\frac{\red{(l(P)^{-d}\int_{U_P} |f_P - f(\xi)|\, d\xi)}^pl(P)^d}{\phi(l(P))^\frac pq}\sum_{R: P\in\Sh^2 (R)} \phi(l(R))^{\frac pq -p}.
	\end{align*}
	Furthermore, Lemma \ref{lem:phil} and Jensen's inequality give
	\begin{align}
	\CJ &\lesssim \sum_{P\in \mW}\frac{\red{(l(P)^{-d}\int_{U_P} |f_P - f(\xi)|\, d\xi)}^pl(P)^d}{\phi(l(P))^p}\nonumber\\ &\lesssim \sum_{P\in \mW}\int_{U_P}\frac{|f_P - f(\xi)|^p}{\phi(l(P))^p}\, d\xi\label{same}\\
	&\leq \sum_{P\in \mW}\int_{U_P}\Bigg(\int_P\frac{|f(\zeta) - f(\xi)|^q}{l(P)^d\phi(l(P))^q}\, d\zeta\Bigg)^{\frac pq}\, d\xi.\nonumber
	\end{align}
	Since $U_P \subseteq 5P$ we have finished estimating $\CJ$.
	
	Now we proceed with \CD. By H\"{o}lder's inequality
	\begin{align*}
	\CD &= \sum_{R\in\mW} \frac{l(R)^{d(1-\frac pq)}}{\phi(l(R))^p}\Bigg(\sum_{S\in\Sh(R)}\int_S |f_S - f(\xi)|^q \, d\xi\frac{l(S)^{d(1 - \frac qp)}}{l(S)^{d(1 - \frac qp)}}\Bigg)^{\frac pq}\\
	&\leq \sum_{R\in\mW} \frac{l(R)^{d(1-\frac pq)}}{\phi(l(R))^p}\Bigg(\sum_{S\in\Sh(R)} l(S)^d\Bigg)^{\frac pq - 1}\sum_{S\in\Sh(R)}\frac{(\int_S|f_S - f(\xi)|^q \, d\xi)^{\frac pq}}{l(S)^{d(\frac pq - 1)}}\\
	&\lesssim \sum_{R\in\mW}\sum_{S\in\Sh(R)} \frac{(\int_S|f_S - f(\xi)|^q \, d\xi)^{\frac pq}}{l(S)^{d(\frac pq - 1)}\phi(l(R))^p}.
	\end{align*}
	By rearranging and using Lemma \ref{lem:phil} we obtain 
	\begin{align*}
	\CD &\lesssim\sum_{S\in \mW}\frac{(\int_S|f_S - f(\xi)|^q \, d\xi)^{\frac pq}}{l(S)^{d(\frac pq - 1)}} \sum_{R: S\in\Sh(R)} \phi(l(R))^{-p}\\ &\lesssim \sum_{S\in \mW}\Bigg(\int_S\frac{|f_S - f(\xi)|^q }{l(S)^d}\, d\xi\Bigg)^{\frac pq} \frac {l(S)^d}{\phi(l(S))^p}.
	\end{align*}
	Hence, by Jensen's inequality,
	\begin{equation*}
	\CD \lesssim \sum_{S\in\mW}\frac {l(S)^d}{\phi(l(S))^p} \int_S\frac {|f_S - f(\xi)|^p}{l(S)^d} \, d\xi  = \sum_{S\in\mW}\int_S \frac {|f_S - f(\xi)|^p}{\phi(l(S))^p}\, d\xi.
	\end{equation*}
	Thus we have arrived at the same situation as in \eqref{same} and the proof is finished (we may need to enlarge the constant $C_\mW$ which can be done by diminishing the cubes in the Whitney decomposition).
\end{proof}
\section{Examples of $\phi$}
\subsection{Positive examples}
We will present some examples of kernels which satisfy \AD\ and \AT.
\begin{example}
	Stable scaling is more than enough for \AD\ to hold. Indeed, if we assume that \red{there exist $\beta_1,\beta_2 \in (0,1)$ for which we have
		$$\lambda^{\beta_1} \lesssim \frac{\phi(\lambda r)}{\phi(r)} \lesssim \lambda^{\beta_2},\quad r>0,\ \lambda \leq 1,$$
		then by the first inequality we get \AT\ and by the second inequality the series in \AD\ are geometric and independent of $r$.}
	
	Let us examine the constant $C_2$ in \AD\ for $p=q=2$, $\alpha\in (0,2)$, and the kernels of the form $K(x,y) = (2-\alpha)|x-y|^{-d-\alpha}$, i.e. $\phi (t) = (2-\alpha)t^{\alpha/2}$. \red{In this case $\frac 1{q-1} = \min(q,p-\frac pq) = 1$ and} for every $r>0$ we have
	$$\sum\limits_{k=1}^\infty \frac{\phi(r)}{\phi(2^kr)} = \red{\sum\limits_{k=1}^\infty \frac{\phi(2^{-k}r)}{\phi(r)} =} \sum\limits_{k=1}^\infty \frac 1{(2^{\alpha/2})^k} = \frac{1}{2^{\alpha/2} - 1}.$$
	This quantity is bounded as $\alpha\to 2^-$. Since the constant in \AT\ is also bounded in this case, we get that the comparability in Theorem \ref{th:main} is uniform  for $\alpha \in (\eps,2)$ for every $\eps>0$.
\end{example}
\begin{example}
	Assume that $\Omega$ is bounded. Let \red{$\gamma\in (0,1)$, $\phi(r) = [\log(1+r)]^\gamma$} and let $R=\diam(\Omega)$. \red{Note that for $r > 0$ we have
		$$1\leq \frac {\log(1+2r)}{\log(1+r)}\leq 2.$$
		Indeed, by looking at the derivative we see that the ratio is decreasing thus the inequalities result from its limits at $0^+$ and at $\infty$. Therefore, $\phi$ satisfies \AT. Furthermore for $r<R$ the lower bound can be replaced with a constant $C = C(R) >1$, hence both series in \AD\ become geometric thus it is satisfied.}
\end{example}
\subsection{O-regularly varying functions}\label{sec:zal}
\begin{definition}
	We say that $\phi$ is O-regularly varying at infinity if there exist $a, b\in \mR$ and $A,B,R>0$ such that
	\begin{equation}
	A\bigg(\frac{r_2}{r_1}\bigg)^a \leq \frac{\phi(r_2)}{\phi(r_1)} \leq B\bigg(\frac{r_2}{r_1}\bigg)^b\label{eq:oreg}
	\end{equation}
	holds whenever $R<r_1<r_2$. Analogously, $\phi$ is O-regularly varying at zero if \eqref{eq:oreg} holds for $0<r_1<r_2<R$. The supremum of $a$ and the infimum of $b$ for which \eqref{eq:oreg} is satisfied are called lower, respectively upper, Matuszewska indexes (or lower/upper indexes). 
\end{definition}
\red{A nice short review of the O-regularly varying functions can be found in the work of Grzywny and Kwa\'s{}nicki \cite[Appendix A]{GRZYWNY20181}, for further reading we refer to the book by Bingham, Goldie, and Teugels \cite{bingham}.}

Assume \textbf{A2} and \textbf{A3}. \red{We will show that the assumptions enforce O-regular variation with positive lower index at 0 and, for unbounded $\Omega$, at infinity by using Proposition A.1 of \cite{GRZYWNY20181}.  Note that by \AT\ for $r>0$, $k\in \mathbb{Z}$, and $z\in [2^{k-1}r,2^{k}r]$ we have $\phi(z) \approx \phi(2^{k}r)$.}

\red{We first consider the regular variation at zero using \cite[Proposition A.1 (c)]{GRZYWNY20181}. Let $R = \diam(\Omega)$ and $t_2 = \frac 1{q-1}$. Then, for every $r\in (0,R)$ and $\eta \in \mR$ we have
	\begin{align*}
	\int_0^r z^{-\eta}\phi(z)^{t_2}\frac{dz}{z} \approx \sum_{k=1}^\infty \phi(2^{-k}r)^{t_2} (2^{-k}r)^{-\eta} = r^{-\eta}\phi(r)^{t_2}\sum_{k=1}^\infty \frac{\phi(2^{-k}r)^{t_2}}{\phi(r)^{t_2}} 2^{k\eta}.
	\end{align*}
	By \AD\ the latter sum is finite for $\eta \leq 0$, it is also bounded away from 0 because of \AT. Therefore we obtain that $\phi^{t_2}$ (and thus, also $\phi$) has to be O-regularly varying at 0 with some lower index $a_0>0$, that is
	$$\frac{\phi(r_2)}{\phi(r_1)} \gtrsim \bigg(\frac {r_2}{r_1}\bigg)^{a_0/t_2},\quad 0<r_1\leq r_2\leq R.$$
	The above condition yields a power-type decay at 0 for $\phi$. This could also be obtained using the other summation condition from \AD\ by applying \cite[Proposition A.1 (d)]{GRZYWNY20181}.}

\red{The behavior of $\phi$ at infinity only comes into play when $\Omega$ is unbounded, thus we assume that $\diam (\Omega) = \infty$ for the remainder of this subsection. Let $r>0$, $\eta \in \mR$, and $t_1 = \min(q,p - \frac pq)$. We have 
	\begin{align*}
	\int_r^\infty z^{-\eta} \phi(z)^{-t_1} \frac{dz}{z} \approx \sum_{k=1}^\infty \phi(2^kr)^{-t_1}(2^kr)^{-\eta} = r^{-\eta}\phi(r)^{-t_1}\sum_{k=1}^\infty \frac{\phi(r)^{t_1}}{\phi(2^kr)^{t_1}}2^{-k\eta}.
	\end{align*}
	By \AD\ and \AT\ the sum is finite and bounded away from 0 if $\eta \geq 0$. Thus $\phi^{-t_1}$ is O-regularly varying at infinity with upper index $-a_\infty<0$, which is equivalent to the O-regular variation with lower index $a_\infty$ for $\phi^{t_1}$:
	$$\frac{\phi(r_2)}{\phi(r_1)} \gtrsim \bigg(\frac {r_2}{r_1}\bigg)^{a_\infty/t_1}, \quad R<r_1\leq r_2<\infty.$$}

\subsection{Negative examples}
We will show some examples for which the seminorms \eqref{eq:fullsemi} and \eqref{eq:truncated} are not comparable. Assume for clarity that $p=q=2$.
\begin{example}
	Let $\Omega = (0,1)\subset \mR$, and let $K(x,y)  \equiv 1$. Consider the function $f(x) = x^{-\gamma}$ with $\gamma\in (0,\frac 12)$. A direct calculation shows that
	\begin{equation}\label{eq:exwhole}\int_0^1\int_0^1 (f(x)-f(y))^2 \, dy \, dx = 2\bigg(\frac 1{1-2\gamma} - \frac 1{(1-\gamma)^2}\bigg).
	\end{equation}
	In particular, $f$ belongs to the corresponding Sobolev space (actually the "Sobolev space" is $L^2(\Omega)$ in this case).
	Let $\eps \in (0,1)$. We have
	\begin{align}\nonumber
	&\int_0^1\int_{x-\eps\delta(x)}^{x+\eps\delta(x)}(f(x) - f(y))^2\, dy \, dx\leq \int_0^1\int_{x(1-\eps)}^{x(1+\eps)}(f(x) - f(y))^2\, dy \, dx \\
	= &\red{\frac {\eps}{1-\gamma} - \frac{(1+\eps)^{1-\gamma} - (1-\eps)^{1-\gamma}}{(1-\gamma)^2} + \frac {(1+\eps)^{1-2\gamma} - (1-\eps)^{1-2\gamma}}{(1-2\gamma)(2-2\gamma)} .}\label{eq:excut}
	\end{align}
	As $\gamma \to \frac 12^-$ the ratio of the right hand side of \eqref{eq:exwhole} and \eqref{eq:excut} goes to infinity which shows that in this case the result from Theorem \ref{th:main} does not hold.
\end{example} 
\begin{example}
	The preceding example gives an idea on how to show an analogous fact for any nonzero $K$ such that $K(0,\cdot)\in L^1([0,1])$.
	On the restricted domain of integration we have $x\approx y$. Therefore $|\frac 1{x^\gamma} - \frac 1{y^\gamma}| \lesssim \frac 1{x^\gamma}$, hence
	\begin{align} \nonumber
	&\int_0^1\int_{B(x,\varepsilon \delta(x))} \Big(\frac 1{x^\gamma} - \frac 1{y^{\gamma}}\Big)^2 K(x,y)\, dy \, dx\\\label{eq:trl1} \lesssim &\int_0^1 \frac 1{x^{2\gamma}} \int_{B(x,\eps\delta(x))} K(x,y) \, dy \, dx.
	\end{align}
	On the other hand, since $K$ is nontrivial, there exists $\eta>0$ such that for every $x\in (0,\eta)$ we have $\int_\eta^1 K(x,y) \, dy \geq C > 0$. Therefore
	\begin{align}
	\int_0^1 \int_0^1 \Big(\frac 1{x^\gamma} - \frac 1{y^{\gamma}}\Big)^2 K(x,y) \, dy \, dx &\geq \int_0^{\eta/2} \int_\eta^1 \Big(\frac 1{x^\gamma} - \frac 1{\eta^{\gamma}}\Big)^2 K(x,y) \, dy \, dx\nonumber\\ &\gtrsim \int_0^{\eta/2} \frac 1 {x^{2\gamma}} \int_{\eta}^{1} K(x,y) \, dy \, dx\nonumber\\
	&\gtrsim \int_0^{\eta/2} \frac 1{x^{2\gamma}} \, dx.\label{eq:whl1}
	\end{align}
	Note that \eqref{eq:trl1} is of the form $\int_0^1 \frac {f(x)}{x^{2\gamma}} \, dx$ with $f(x)$ bounded and $\lim\limits_{x\to 0^+} f(x) = 0$. Let us fix an arbitrarily small $\xi > 0$, and let $\rho$ be sufficiently small so that $f(x)\leq \xi$ for $x\in(0,\rho)$. If we separate $\int_0^1 = \int_0^\rho + \int_\rho^1$, then we see that the ratio of \eqref{eq:trl1} and \eqref{eq:whl1} tends to 0 as $\gamma \to \frac 12$.
\end{example}
\begin{remark}
	In previous examples the kernel was integrable. This means that
	\begin{align*}\int_\Omega\int_\Omega (f(x) - f(y))^2 K(x,y) \, \red{dy \, dx} \leq 2\int_\Omega \int_\Omega f(x)^2 K(x,y) \, dy \, dx \leq 2\|f\|_{L^2(\Omega)}^2\|K(0,\cdot)\|_{L^1(\mR^d)}.
	\end{align*}
	Therefore, even though the quadratic forms \eqref{eq:fullsemi} and \eqref{eq:truncated} are incomparable, the Triebel--Lizorkin norm $\|\cdot\|_{F_{p,q}(\Omega)}$ is comparable when we replace the full seminorm with the truncated one.
\end{remark}
\begin{example}
	For $K(x,y) = |x-y|^{-1}$ on $\Omega=(0,1)$ the seminorms also fail to be comparable. Consider the functions $f_n(x) = n \wedge \frac 1x$. Since $$\int_0^1\int_0^{\red{x}} (f(x) - f(y))^2K(x,y) \red{\, dy \, dx} = \frac 12 \int_0^1\int_0^1 (f(x) - f(y))^2K(x,y) \red{\, dy \, dx},$$
	we will assume that \red{$y<x$}, and work only with the integral on the left hand side. Note that for $f_n$, the integral over $(0,\frac 1n)^{2}$ vanishes. We split as follows
	\begin{align}
	\int_0^1 \int_0^{\red{x}} (f_n(x) - f_n(y))^2 K(x,y) \red{\, dy \, dx} &= \int_{1/n}^{1}\int_{1/n}^{\red{x}} \bigg(\frac 1x - \frac 1y\bigg)^2 K(x,y)\, \red{dy \, dx}\label{eq:hil1}\\
	&+ \int_{1/n}^1\int_0^{1/n}\bigg(n - \red{\frac 1x}\bigg)^2K(x,y) \red{\, dy \, dx}. \label{eq:hil2}
	\end{align}
	We first compute the right hand side of \eqref{eq:hil1}. Note that the integrand is equal to $\frac {(\red{x}-\red{y})^2}{\red{y}^2\red{x}^2}\cdot\frac{1}{\red{x}-\red{y}} = \frac {\red{x}-\red{y}}{\red{y}^2\red{x}^2}$.
	\begin{align*}
	\int_{1/n}^1 \int_{1/n}^{\red{x}} \frac{\red{x}-\red{y}}{\red{y}^2\red{x}^2} \, d\red{y} \, d\red{x} = 	\int_{1/n}^1 \int_{1/n}^{\red{x}} \frac{1}{\red{y}^2\red{x}} \, d\red{y} \, d\red{x} - 	\int_{1/n}^1 \int_{1/n}^{\red{x}}\frac{1}{\red{y}\red{x}^2} \, d\red{y} \, d\red{x}= n\log n - 2n + \log n +2.
	\end{align*}
	For \eqref{eq:hil2} we only show the asymptotics.
	\begin{align*}
	&\int_{1/n}^1\int_0^{1/n}\bigg(n - \frac 1{\red{x}}\bigg)^2K(x,y) \, d\red{y} \, d\red{x}\\ = &\int_{1/n}^1 \bigg(n - \frac 1{\red{x}}\bigg)^2\bigg(\log \red{x} - \log \bigg(\red{x}-\frac 1n\bigg)\bigg)\, d\red{x}\\
	= &-n^2\int_{1/n}^1\bigg(1 - \frac 1{n\red{x}}\bigg)^2\log\bigg(1 - \frac 1{n\red{x}}\bigg) \, d\red{x}\\ = &-n\int_0^{1-1/n}\frac{t^2}{(1-t)^2}\log t\, dt.\\
	\end{align*}
	For $n>2$ we split the latter integral: $\int_0^{1-1/n} = \int_0^{1/2} + \int_{1/2}^{1-1/n}$. The first one converges, i.e. it is a (negative) constant. In the second one $t^2 \approx 1$, and $\frac{\log t}{1-t} \approx -1$, therefore
	\begin{equation}
	-n\int_0^{1-1/n}\frac{t^2}{(1-t)^2}\log t\, dt \approx n\bigg(1 + \int_{1/2}^{1-1/n} \frac {dt}{1-t}\bigg) = n(1 + \log n - \log 2).\label{eq:t1mt}
	\end{equation}
	Thus we get the asymptotics
	\begin{align}
	\int_0^1 \int_0^1 (f_n(x) - f_n(y))^2 K(x,y) \, d\red{y} \, d\red{x} \approx n\log n. \label{eq:fulhil}
	\end{align}
	Now consider the truncated case. For clarity, assume that $\epsilon = \frac 12$. 
	\begin{align}
	\int_0^1 \int_{\red{x}/2}^{\red{x}} (f_n(x) - f_n(y))^2 K(x,y) \, d\red{y} \, d\red{x} &= \int_{2/n}^{1}\int_{\red{x}/2}^{\red{x}} \bigg(\frac 1x - \frac 1y\bigg)^2 K(x,y)\, d\red{y} \, d\red{x}\label{eq:trhil1}\\
	&+ \int_{1/n}^{2/n}\int_{1/n}^{\red{x}}\bigg(\frac 1x - \frac 1y\bigg)^2K(x,y) \, d\red{y} \, d\red{x} \label{eq:trhil2}\\
	&+ \int_{1/n}^{2/n}\int_{\red{x}/2}^{1/n}\bigg(n - \frac 1{\red{x}}\bigg)^2K(x,y) \, d\red{y} \, d\red{x}.\label{eq:trhil3}
	\end{align}
	For the right hand side of \eqref{eq:trhil1} and \eqref{eq:trhil2} we note that
	\begin{equation*}
	\int_{2/n}^{1}\int_{\red{x}/2}^{\red{x}} \bigg(\frac 1x - \frac 1y\bigg)^2 K(x,y)\, d\red{y} \, d\red{x} \leq \int_{2/n}^1\int_{\red{x}/2}^{\red{x}} \frac 1{\red{y}^2\red{x}} \, d\red{y} \, d\red{x} = \frac n2 - 1,
	\end{equation*}
	and
	\begin{equation*}
	\int_{1/n}^{2/n}\int_{1/n}^{\red{x}}\bigg(\frac 1x - \frac 1y\bigg)^2K(x,y) \, d\red{y} \, d\red{x} \leq \int_{1/n}^{2/n}\int_{1/n}^{\red{x}} \frac 1{\red{y}^2\red{x}} \, d\red{y} \, d\red{x} = n\log 2 - \frac n2.
	\end{equation*}
	The last integral \eqref{eq:trhil3} is estimated as follows
	\begin{align*}
	&\int_{1/n}^{2/n}\int_{\red{x}/2}^{1/n}\bigg(n - \frac 1{\red{x}}\bigg)^2K(x,y) \, d\red{y} \, d\red{x}\\
	 = &\int_{1/n}^{2/n} \bigg(n - \frac 1{\red{x}}\bigg)^2\bigg(\log \frac {\red{x}}2 - \log \bigg(\red{x} - \frac 1n\bigg)\bigg) \, d\red{x}\\
	= &-n^2 \int_{1/n}^{2/n} \bigg(1 - \frac 1{n\red{x}}\bigg)^2 \bigg(\log \bigg(1 - \frac 1{n\red{x}}\bigg) + \log 2 \bigg) \, d\red{x} \\\leq
	&-n\int_0^{1/2}\frac{t^2}{(1-t)^2}\log t\, dt \approx n.
	\end{align*}
	To conclude, we get
	\begin{equation}\label{eq:trhil}
	\int_0^1\int_{B(\red{x},\delta(\red{x})/2)} (f_n(x) - f_n(y))^2 K(x,y) \, d\red{y} \, d\red{x} \lesssim n.
	\end{equation}
	Since the ratio of the right hand sides of \eqref{eq:fulhil} and \eqref{eq:trhil} diverges as $n\to\infty$, our claim is proven.
\end{example}
\section{The 0-order kernel}
\begin{theorem}\label{th:0order}
	Let $\Omega$ be a bounded uniform domain. Then, if $1< q\leq p<\infty$, then for every $0<\theta\leq 1$
	\begin{align}
	&\bigg(\int_{\Omega}\bigg(\int_{\Omega} \frac{|f(x) - f(y)|^q}{|x-y|^d}\, dy\bigg)^{\frac pq}\, dx\bigg)^{\frac 1p}\label{eq:0order}\\ \lesssim &\,\bigg(\int_{\Omega}\bigg(\int_{B(x,\theta\delta(x))} \frac{|f(x) - f(y)|^q}{|x-y|^d}(\big|\hspace{-1pt}\log|x-y|\big|\,\vee\, 1)^{\red{q}}\, dy\bigg)^{\frac pq}\, dx\bigg)^{\frac 1p}.\label{eq:0orderlog}
	\end{align}
	\red{The constant in the inequality depends only on $p,q,\theta,\Omega$.}
\end{theorem}
In order to obtain this result we first prove an analogue of Lemma \ref{lemMaximal} for $K(x,y) = |x-y|^{-d}$, i.e. $\phi \equiv 1$. For now every integral is restricted to $\Omega$ by default.
\begin{lemma}\label{lem:max2}
	Let $\Omega$ be a bounded domain with Whitney covering $\mW$. Assume that $g\in L^1_{loc}(\mR^d)$, and $0<r<\diam(\Omega)$. Then for every $Q\in\mW$ and $x\in\Omega$ we have
	\begin{align}\label{eq:max2far}
	\int_{|y-x|>r} \frac{g(y) \, dy}{|y-x|^d}&\lesssim Mg(x)(|\log r|\red{\vee 1}),\\
	\sum_{S:D(Q,S)>r}  \frac{\int_S g(y) \, dy}{D(Q,S)^d}&\lesssim \inf_{x\in Q} Mg(x)(|\log r|\red{\vee 1}).\label{eq:max2far2}
	\end{align}
	and
	\begin{equation}\label{eq:max2allover}
	\sum_{S\in\mathcal{W}} \frac{l(S)^d}{D(Q,S)^d} \lesssim |\log \red{l(Q)}|\red{\vee 1}.
	\end{equation}
\end{lemma}
\begin{proof}
	Let $x\in \Omega$. If we take $R = \diam(\Omega)$, then proceeding as in Lemma \ref{lemMaximal} we get
	\begin{align*}
	\int_{|y-x|\red{>}r} \frac{g(y)\, dy}{|y-x|^d} \leq \sum_{k=1}^{\lceil\log_2(R/r)\rceil} \int_{2^{k-1}r\leq|y-x|<2^k r} \frac{g(y)\,  dy}{|x-y|^d} \lesssim Mg(x) \lceil\log_2(R/r)\rceil \lesssim Mg(x)(|\log r|\red{\vee 1}).
	\end{align*}
	As in the proof of Lemma \ref{lemMaximal}, in order to prove \eqref{eq:max2far2} we use \eqref{eq:max2far}, and we are left with
	\begin{equation*}
	\int_{|x-y|<r} \frac {g(y)\, dy}{(|x-y| + r)^d} \lesssim \frac 1{|B(x,r)|} \int_{B(x,r)} g(y) \, dy \leq Mg(x)(|\log r|\red{\vee 1}).
	\end{equation*}
	Finally, \eqref{eq:max2allover} is obtained by taking $r = l(Q)$ and $g\equiv 1$.
\end{proof}
\red{We will also use the following result similar to Lemma \ref{lem:phil}.
	\begin{lemma}\label{lem:phil0}
		Let $\Omega$ be a bounded uniform domain with admissible Whitney decomposition $\mW$ and let $\rho >0$ and $\eta>1$. Then, for every $S\in \mW$ we have
		\begin{equation}\label{eq:phil1}
		\sum_{R: S\in \Sh_\rho(R)} 1 \lesssim |\log l(S)|\vee 1.
		\end{equation}
		If $S\in \Sh_\rho(R)$, then
		\begin{equation}\label{eq:phil2}
		\sum_{P\in [S,R)} (|\log l(P)|\vee 1)^{-\eta} \lesssim (|\log l(R)|\vee 1)^{1-\eta}.
		\end{equation}
		Furthermore, for every $P\in \mW$
		\begin{equation}\label{eq:phil3}
		\sum_{R: P\in \Sh_\rho(R)} (|\log l(R)|\vee 1)^{-\eta} \lesssim 1.
		\end{equation}
\end{lemma}}
\begin{proof}
	\red{Throughout the proof we let $l(S) = 2^{s_0}$, $l(R) = 2^{r_0}$, $l(P) = 2^{p_0}$, whenever the cubes are fixed.}
	
	\red{Arguing as in the proof of Lemma \ref{lem:phil} we get that there is a limited number of cubes of a given side length contributing to the sum in \eqref{eq:phil1} and the smallest of these cubes must have side length at least $2^{s_0 - l_0}$ for some fixed natural number $l_0 \geq 0$. Therefore, if we let $2^{m_0}$ be the side length of the largest cube in $\mW$, then we have
		\begin{equation*}
		\sum_{R: S\in \Sh_\rho(R)} 1 \lesssim \sum_{k=s_0-l_0}^{{m_0}} 1 = {m_0} - s_0 + l_0 + 1 \approx |\log l(S)| \vee 1.
		\end{equation*}
		As in Lemma \ref{lem:phil}, in \eqref{eq:phil2} we have limited number of cubes of the same length and the cube length cannot be larger than $2^{r_0+l_0}$ and smaller than $2^{s_0 - l_0}$ ($l_0$ may be different than above, but it does not depend on $S$ and $R$). Therefore we estimate the sum in \eqref{eq:phil2} as follows:
		\begin{align*}
		\sum_{P\in [S,R)} (|\log l(P)|\vee 1)^{-\eta} \lesssim \sum_{k=s_0-l_0}^{r_0 + l_0} (|k|\vee 1)^{-\eta} \leq \sum_{k=-\infty}^{r_0+l_0} (|k|\vee 1)^{-\eta}.
		\end{align*}
		Since $\eta > 1$, the latter series is finite and it is of order $(|r_0|\vee 1)^{1-\eta}$, which proves \eqref{eq:phil2}.}
	
	\red{In order to prove \eqref{eq:phil3} we argue as above in terms of the numbers of the cubes, and because of $\eta>1$ we get
		\begin{align*}
		\sum_{R: P\in \Sh_\rho(R)} (|\log l(R)|\vee 1)^{-\eta} \lesssim \sum_{k = p_0 - l_0}^{{m_0}} (|k|\vee 1)^{-\eta} \leq \sum_{k=-\infty}^{{m_0}} (|k|\vee 1)^{-\eta} = C.
		\end{align*}}
\end{proof}
\begin{proof}[Proof of Theorem \ref{th:0order}]
	\red{We proceed as in Theorem \ref{th:main} starting with $1$ in place of $\phi$. The integrals over $Q\times 2Q$ are trivially estimated, because the kernel in \eqref{eq:0orderlog} is larger than the one in \eqref{eq:0order}.}
	
	\red{In \AAA\ and \BB\ the modification is quite straightforward. Lemma \ref{lemMaximal} is used in \eqref{eq:repA} and \eqref{eq:repB} respectively. Using Lemma \ref{lem:max2} instead, we get respectively $(|\log l(Q)|\vee 1)^{\frac 1q}$ and $(|\log l(P)|\vee 1)^{\frac 1q}$. The remaining arguments are conducted with $(|\log r|\vee 1)^{-\frac 1q}$ in place of $\phi(r)$. Note that $(|\log r|\vee 1)^{-\frac 1q} \approx (|\log 2r|\vee 1)^{-\frac 1q}$. We remark that this yields estimates for \AAA\ and \BB\ which are better than the ones in the statement, in fact both expressions are bounded from above by
		\begin{equation}
		\bigg(\int_{\Omega}\bigg(\int_{B(x,\theta\delta(x))} \frac{|f(x) - f(y)|^q}{|x-y|^d}(\big|\hspace{-1pt}\log|x-y|\big|\,\vee\, 1)\, dy\bigg)^{\frac pq}\, dx\bigg)^{\frac 1p}.\label{eq:log0}
		\end{equation}
		Notice the lack of exponent $q$ in the logarithmic term. At this point we distinguish between the case $p=q$ and $p\neq q$. In the former case the test functions $g$ from \eqref{eq:dual} are defined by the condition $\int_\Omega\int_\Omega g(x,y)^{p'}\, dy\, dx \leq~1$, therefore \CC\ can be estimated exactly as \AAA\ and \BB\ because we can interchange the roles of $Q,S$ and $x,y$ using Tonelli's theorem. Thus, in this case we in fact obtain an estimate better than postulated, as the whole expression in \eqref{eq:0order} is approximately bounded from above by \eqref{eq:log0}.}
	
	\red{For the remainder of the proof we assume that $p > q$. The procedure for \CC\ is also similar to the one in the proof of Theorem \ref{th:main}, but the computations are slightly different in terms of the exponents, therefore we give the details. There are no essential changes up to the moment of splitting into \CJ\ and \CD, thus we make it our starting point. As in the proof of Theorem \ref{th:main} we get
		\begin{align*}
		\CD &= \sum_{R\in\mW} l(R)^{d(1-\frac pq)}\bigg(\sum_{S\in\Sh(R)}\int_S |f_S - f(\xi)|^q\, d\xi \frac{l(S)^{d(1-\frac qp)}}{l(S)^{d(1-\frac qp)}}\bigg)^{\frac pq}\\
		&\lesssim \sum_{R\in\mW}\sum_{S\in\Sh(R)} l(S)^{d(1-\frac pq )}\bigg(\int_S|f_S - f(\xi)|^q\,d\xi\bigg)^{\frac pq}.
		\end{align*}
		We rearrange, use \eqref{eq:phil1} and then Jensen's inequality twice to obtain:
		\begin{align*}
		\CD &\lesssim \sum_{S\in \mW}l(S)^{d(1-\frac pq)}\bigg(\int_S|f_S - f(\xi)|^q\,d\xi\bigg)^{\frac pq}\bigg(\sum_{R: S\in \Sh(R)} 1\bigg)\\
		&\lesssim \sum_{S\in \mW} l(S)^{d}(|\log l(S)|\vee 1)\bigg(\frac{1}{l(S)^d}\int_S|f_S - f(\xi)|^q\, d\xi\bigg)^{\frac pq}\\
		&\leq \sum_{S\in \mW} (|\log l(S)|\vee 1)\int_S |f_S - f(\xi)|^p\, d\xi\\
		&\leq \sum_{S\in\mW}\int_S \bigg(\int_S \frac{|f(\zeta) - f(\xi)|^q}{l(S)^d}(|\log l(S)|\vee 1)^{\frac qp}\, d\zeta\bigg)^{\frac pq}\, d\xi,
		\end{align*}
		and thus \CD\ is estimated, since $\frac qp < 1 < q$.}
	
	\red{In order to estimate \CJ\ we write $|f_P - f_{\mN(P)}| = |f_P - f_{\mN(P)}|\frac{|\log l(P)|\vee 1}{|\log l(P)|\vee 1}$ and we use H\"older's inequality with exponent $q$ and \eqref{eq:phil2}:
		\begin{align*}
		\CJ \leq &\hspace{-2.04pt}\sum_{R\in \mW} l(R)^{d(1-\frac pq)}\bigg[\sum_{S\in\Sh(R)}\bigg(\sum_{P\in [S,R)}|f_P - f_{\mN(P)}|^q(|\log l(P)|\vee 1)^q l(S)^d\bigg)\bigg(\sum_{P\in [S,R)}(|\log l(P)|\vee 1)^{-q'}\bigg)^{\frac q{q'}}\bigg]^{\frac pq}\\
		\lesssim &\hspace{-2.04pt}\sum_{R\in \mW}l(R)^{d(1-\frac pq)}(|\log l(R)|\vee 1)^{-\frac pq}\bigg(\sum_{S\in\Sh(R)} \sum_{P \in [S,R)} |f_P - f_{\mN(P)}|^q(|\log l(P)|\vee 1)^{q}l(S)^d\bigg)^{\frac pq}.
		\end{align*}
		By rearranging as in the proof of Theorem \ref{th:main} and by using H\"older's and Jensen's inequalities we further estimate \CJ\ from above by a multiple of
		\begin{align*}
		 &\sum_{R\in\mW} l(R)^{d(1 - \frac pq)} (|\log l(R)|\vee 1)^{-\frac pq} \bigg(\sum_{P\in\Sh^2(R)}\sum_{S\in \Sh^2(P)} \big(\int_{U_P} \frac{|f_P - f(\xi)|}{l(P)^d}\, d\xi\big)^q(|\log l(P)|\vee 1)^q l(S)^d\bigg)^{\frac pq}\\
		\lesssim &\sum_{R\in\mW} l(R)^{d(1 - \frac pq)} (|\log l(R)|\vee 1)^{-\frac pq}\bigg(\sum_{P\in\Sh^2(R)} \big(\int_{U_P} \frac{|f_P - f(\xi)|}{l(P)^d}\, d\xi\big)^q(|\log l(P)|\vee 1)^q l(P)^d\bigg)^{\frac pq}\\
		\leq &\sum_{R\in\mW}\sum_{P\in\Sh^2(R)} (|\log l(R)|\vee 1)^{-\frac pq} (|\log l(P)|\vee 1)^p\int_{U_P} |f_P - f(\xi)|^p\, d\xi.
		\end{align*}
		We rearrange once more and use \eqref{eq:phil3} (recall that $p>q$) and Jensen's inequality to get that, up to a multiplicative constant, \CJ\ does not exceed
		\begin{align*}
		&\sum_{P\in \mW}(|\log l(P)|\vee 1)^p\int_{U_P} |f_P - f(\xi)|^p\, d\xi\bigg(\sum_{R: P\in \Sh^2(R)} (|\log l(R)|\vee 1)^{-\frac pq}\bigg)\\
		\lesssim &\sum_{P\in \mW}(|\log l(P)|\vee 1)^p\int_{U_P} |f_P - f(\xi)|^p\, d\xi\\
		\lesssim &\sum_{P \in \mW} \int_{U_P}\bigg(\int_P\frac{|f(\zeta) - f(\xi)|^q}{l(P)^d}(|\log l(P)|\vee 1)^q\,d\zeta\bigg)^{\frac pq}\,d\xi.
		\end{align*}
		This finishes the proof.}
\end{proof}
Since the kernel in \eqref{eq:0orderlog} is significantly larger than the one in \eqref{eq:0order}, it is plausible that the converse inequality is not true. We will show the existence of a counterexample when $\Omega = (0,1)$, $p=q=2$. For an open interval $I\subseteq \mR$ we let $$F_0(I) = \bigg\{f \in L^2(I): \int_I\int_I \frac{(f(x)-f(y))^2}{|x-y|}\, dy \, dx < \infty\bigg\},$$
$$F_{\log}(I) = \bigg\{f\in L^2(I):\int_I\int_I\frac{(f(x)-f(y))^2}{|x-y|} (|\log|x-y|| \vee 1) \, dy \, dx <\infty \bigg\}.$$
\red{We note that in $F_{\log}(I)$ the logarithm is in power 1. This suffices for our present purpose, because $q>1$ in Theorem \ref{th:0order}.}
\begin{theorem}
	For every $\theta\in (0,1]$, there exists $f \in F_0(0,1)\cap L^\infty(0,1)$ such that
	\begin{equation}\label{eq:lognorm}
	\int_0^1 \int_{B(x,\theta \delta(x))} (f(x) - f(y))^2 |x-y|^{-1}(|\log|x-y||\vee 1) \, dy \, dx = \infty.
	\end{equation}
\end{theorem}
\begin{proof}$ $
	
	\textbf{Step 1.} First, note that the finiteness of the left hand side of \eqref{eq:lognorm} implies that $f\in F_{\log}(\frac {n}{2n+1}, \frac{n+1}{2n+1})$ for a sufficiently large $n\in \mathbb{N}$. Indeed, if $\theta \geq \frac 1n$ for some natural number $n \geq 2$, then \begin{align}\label{eq:ciag}\int_0^1\int_{B(x,\theta\delta(x))} (\ldots)\geq \int_0^1 \int_{B(x,\delta(x)/n)} (\ldots) \geq \int_{\frac {n}{2n+1}}^{\frac {n+1}{2n+1}}\int_{B\big(x,\frac 1{2n+1}\big)} (\ldots) \geq \int_{\frac {n}{2n+1}}^{\frac {n+1}{2n+1}}\int_{\frac {n}{2n+1}}^{\frac {n+1}{2n+1}} (\ldots).\nonumber\end{align}
	We fix a number $n$ for which \eqref{eq:ciag} is satisfied.
	%
	
	\textbf{Step 2.} In order to construct the counterexample we will use the asymptotics of the Fourier expansions of functions in $F_0(I)$ and $F_{\log}(I)$. \red{We adopt the following convention for the Fourier coefficients of an integrable function $f$ on an interval $(a,b)$:
		\begin{equation*}
		\widehat{f}(m) = \frac{1}{b-a}\int_a^b f(x) e^{-\frac{2\pi imx}{b-a}}\, dx,\quad m\in \mathbb{Z}.
		\end{equation*}
		Below, by $\widehat{f}(m)$ we mean the Fourier coefficient on $(0,1)$.} Let $f$ satisfy $f(x+1) = f(x)$ for $x\in \mR$. Let $K(x,y)$ be equal to $|x-y|^{-1}$ (resp. $|x-y|^{-1}(|\log|x-y||\vee 1)$). \red{We claim that given $f\in L^\infty(0,1)$, it belongs to $F_0(0,1)$} (resp. $F_{\log} (0,1)$) if and only if
	\begin{align*}
	\int_0^1\int_0^1 (f(x) - f(x-h))^2 K(0,h) \, dh \, dx < \infty.
	\end{align*}
	Indeed, we have
	\begin{align*}
	\int_0^1\int_0^1 (f(x) - f(y))^2 K(x,y) \, dy \, dx =\, &2\int_0^1 \int_0^x (f(x) - f(y))^2 K(x,y) \, \red{dy \, dx}\\
	=\, & \red{2}\int_0^1\int_0^x (f(x) - f(x-h))^2 K(0,h) \, dh \, dx.
	\end{align*}
	Therefore, it suffices to verify that $\int_0^1 \int_x^1 (f(x) - f(x-h))^2 K(0,h) \, dh \, dx <\infty$ for bounded $f$. Clearly we can assume that $K(x,y) = |x-y|^{-1}(|\log|x-y||\vee 1)$.
	\begin{align*}&\int_0^1\int_x^1 (f(x) - f(x-h))^2 K(0,h) \, dh \, dx\lesssim \int_0^1\int_x^1 \frac{(-\log h) \vee 1}{h} \, dh \, dx\\
	= &\int_0^{1/e} \int_x^{\red{1/e}} \frac{-\log h}{h} \, dh \, dx + \red{\int_0^{1/e} \int_{1/e}^{1} \frac{1}{h} \, dh \, dx+} \int_{1/e}^1 \int_x^1 \frac{1}{h} \, dh \, dx.
	\end{align*}
	\red{All the} integrals are finite, therefore the claim is proved.
	
	By Parseval's identity and Tonelli's theorem we get
	\begin{align*}
	&\int_0^1K(0,h)\int_0^1 (f(x) - f(x-h))^2 \, \red{dx} \, \red{dh} \\= &\int_0^1 K(0,h) \sum_{m \in \mathbb{Z}} |\widehat{f}(m)|^2 |1 - e^{2\pi i m h}|^2 \, dh\\
	=&\sum_{m\in\mathbb{Z}} |\widehat{f}(m)|^2 \int_0^1 |1 - e^{2\pi i m h}|^2 K(0,h) \, dh\\ = &2\sum_{m\in\mathbb{Z}} |\widehat{f}(m)|^2 \int_0^1 (1-\cos(2\pi m h)) K(0,h) \, dh.
	\end{align*}
	Now let us inspect the remaining integrals for both cases of $K$. For $m\neq 0$ we have
	\begin{align*}
	\int_0^1 \frac{1-\cos (2\pi m h)}{h} \, dh = \int_0^{|m|} \frac{1-\cos (2\pi h)}{h} \, dh \approx \log|m|.
	\end{align*}
	In the logarithmic case
	\begin{align*}
	\int_0^1 \frac{1-\cos(2\pi m h)}{h} (-\log h \vee 1) \, dh = \int_0^{|m|}\frac{1-\cos(2\pi h)}{h} (-\log \frac {h}{|m|} \vee 1) \, dh\approx \log^2|m|.
	\end{align*}
	To summarize, for bounded functions we can characterize $F_0(0,1)$ by
	\begin{equation}\label{eq:0char}
	\sum_{m\in \mathbb{Z}, m \neq 0} |\widehat{f}(m)|^2 \log|m| < \infty
	\end{equation}
	and $F_{\log}(0,1)$ by 
	\begin{equation}\label{eq:logchar}
	\sum_{m\in \mathbb{Z}, m\neq 0} |\widehat{f}(m)|^2 \log^2|m| < \infty.
	\end{equation}
	The same characterizations hold for $I = (\frac {n}{2n+1}, \frac{n+1}{2n+1})$ and the respective Fourier expansion.
	
	\textbf{Step 3.} We give an example of $f\in F_0(0,1)\cap L^\infty(0,1)$ for which \eqref{eq:0char} is satisfied and \eqref{eq:logchar} is not. For $m = (2n+1) 2^l$, $l=1,2,\ldots$, we put $\widehat{f}(m) = \frac 1{l^{3/2}}$. For other $m$ we let $\widehat{f}(m) = 0$. Note that $f$ is $\frac 1{2n+1}$--periodic. Therefore the $j$-th Fourier coefficient of $f$ on $(\frac{n}{2n+1},\frac{n+1}{2n+1})$ is equal to its $(2n+1)\cdot j$-th Fourier coefficient on $(0,1)$. Since $(\widehat{f}(m))_{m\in\mathbb{Z}}$ is summable, $f$ is bounded. Furthermore $l^{-3}\log[ (2n+1)2^l] = l^{-2}\log 2 + l^{-3}\log (2n+1)$ and $l^{-3} \log^2(2^l) \approx l^{-1}$. Therefore \eqref{eq:0char} is satisfied and \eqref{eq:logchar} is not. By \eqref{eq:ciag}, the proof is finished.
	
\end{proof}
\section{Uniformity is not a sharp condition}\label{sec:pasy}
In this section we examine the strip $\mR\times (0,1)$ which is a non-uniform domain. We will show that the comparability fails for fractional Sobolev spaces with $\alpha < 1$. Then we \red{prove} that for $\alpha >1$ and slightly more general kernels the comparability holds. Later, we present a higher-dimensional case \red{in which} the comparability may also hold for $\alpha < 1$ in non-uniform domains. For clarity of the presentation, we assume that $p=q=2$.
\begin{example}
	Let $\Omega = \mR \times (0,1)$ and let $K(x,y) = |x-y|^{-2-\alpha}$. Note that $\Omega$ is not uniform --- if we take two cubes far from each other we will fail to find a sufficiently large central cube in any chain connecting them.
	
	We will show for $\alpha \in (0,1)$ the comparability does not hold. Consider a sequence of functions $(f_n)$ given by the formula $f_n(x_1,x_2) = (1 - \frac{|x_1|}{n})\vee 0$. Since $f_n$ are constant on the second variable, for every $\xi\in (0,1)$ we have
	\begin{align*}
	\int_{\Omega}\int_\Omega \frac{(f_n(x) - f_n(y))^2}{|x-y|^{2+\alpha}} \red{\, dy \, dx} = \int_{\mR}\int_{\mR}(f_n(x_1,\xi) - f_n(y_1,\xi))^2\int_0^1\int_0^1 |x-y|^{-2-\alpha} \red{\, dy_2 \, dx_2 \, dy_1 \, dx_1}.
	\end{align*}
	Let the integral over $(0,1)\times(0,1)$ be called $\kappa(x_1,y_1)$. We claim that $\kappa(x_1,y_1)$ is comparable with $|x_1-y_1|^{-2-\alpha}$ if $|x_1 - y_1| \geq 1$ and with $|x_1-y_1|^{-1-\alpha}$ otherwise. Indeed, we have $|x-y| \approx |x_1 - y_1| + |x_2 - y_2|$. If $|x_1 - y_1| \geq 1$, then
	\begin{equation*}
	\int_0^1\int_0^1 |x-y|^{-2-\alpha} \red{\, dy_2\, dx_2} \approx |x_1-y_1|^{-2-\alpha}\int_0^1\int_0^1 \, \red{dy_2\, dx_2} = |x_1 - y_1|^{-2-\alpha}.
	\end{equation*}
	For $|x_1 - y_1|<1$ note that for fixed $a>0$
	\begin{align*}
	&a^{1+\alpha}\int_0^1 \int_0^1 (a + |x_2-y_2|)^{-2-\alpha} \, \red{dy_2 \, dx_2}\\ \approx\, &a^{1+\alpha}\int_0^1\int_0^{\red{x}_2} (a + \red{x_2 - y_2})^{-2-\alpha} \, \red{dy_2 \, dx_2}\\
	=\, &\frac{a^{1+\alpha}}{\red{1+\alpha}}\int_0^1 (a^{-1-\alpha} - (a+\red{x_2})^{-1-\alpha}) \, \red{dx_2}\\ = &\red{\frac 1{1+\alpha}} - \red{\frac 1{1+\alpha}}\int_0^1 \big(1 + \frac{\red{x_2}}{a}\big)^{-1-\alpha}\, \red{dx_2}.
	\end{align*}
	\red{For $a=|x_1 - y_1|<1$ we have $x_2/a > x_2$, so the latter integral is bounded from above by $C\in (0,1)$.} Thus the whole expression is approximately \red{equal to a positive} constant which proves our claim.
	
	The shape of $\Omega$ grants that for every $\theta\in(0,1]$ we have
	\begin{align*}
	&\int_{\Omega}\int_{B(\red{x,\theta\delta(x)})} \frac{(f_n(x) - f_n(y))^2}{|x-y|^{2+\alpha}} \red{\, dy \, dx}\\ \leq &\int_{\mR}\int_{B(\red{x_1},1)} (f_n(x_1,\xi) - f_n(y_1,\xi))^2 \kappa(x_1,y_1) \, \red{dy_1 \, dx_1}. 
	\end{align*}
	To simplify the notation we will write $f_n(x_1) = f_n(x_1,\xi)$ for some fixed $\xi\in (0,1)$, $x\in\mR$. Since $f_n$ is Lipschitz \red{with constant $1/n$}, we have
	\begin{align*}
	&\int_{\mR}\int_{B(\red{x_1},1)} (f_n(x_1) - f_n(y_1))^2\kappa(x_1,y_1) \, \red{dy_1 \, dx_1}\\ \approx &\int_{\mR} \int_{B(\red{x_1},1)} (f_n(x_1) - f_n(y_1))^2 |x_1-y_1|^{-1-\alpha} \red{\, dy_1 \, dx_1}\\
	=&\int_{-n-1}^{n+1}\int_{B(\red{x_1},1)} (f_n(x_1) - f_n(y_1))^2|x_1-y_1|^{-1-\alpha}\, \red{dy_1 \, dx_1}\\ \lesssim &\frac 1{n^2} \int_{-n-1}^{n+1}\int_{B(\red{x_1},1)} |x_1-y_1|^{1-\alpha}\, \red{dy_1 \, dx_1} \approx \frac 1n.
	\end{align*}
	Thanks to the fact that $\alpha < 1$, the full seminorm is significantly greater as $n\to\infty$:
	\begin{align*}
	&\int_{\mR}\int_{\mR} (f_n(x_1) - f_n(y_1))^2 \kappa(x_1,y_1) \, \red{dy_1 \, dx_1} \gtrsim \int_{-\frac n2}^{0}\int_{-\infty}^{-n} |x_1-y_1|^{-2-\alpha} \red{\, dy_1 \, dx_1}\\
	= &\int_{-\frac n2}^{0} \frac 1{1+\alpha} \frac 1{(\red{x_1}+n)^{1+\alpha}}\, d\red{x_1}
	\geq \frac 1{1+\alpha}\frac{n/2}{\red{n}^{1+\alpha}} \approx \frac 1{n^\alpha}.
	\end{align*}
\end{example}	
\begin{lemma}\label{lem:tech}
	Let $\Omega = \mR\times (0,1)$. If $f\colon \mR^2 \to [0,\infty)$ is radial, then \red{$\int_\Omega (1\vee |x|)f(x) \, dx \approx \int_{\mR^2} f(x) \, dx < \infty$ with a constant independent of $f$.}
\end{lemma}
\begin{proof}
	Note that for $n \in \mathbb{N}$ the area of $\Omega\cap(B_n\setminus B_{n-1})$ is comparable to the $1/n$-th of the area of the annulus $B_n\setminus B_{n-1}$. Therefore by the rotational symmetry of $f$ we get
	\begin{align*}
	\int_\Omega \red{(1\vee |x|)}f(x) \, dx \approx \sum\limits_{n\in \mathbb{N}} \int_{\Omega\cap(B_n\setminus B_{n-1})}nf(x) \, dx &\approx \sum\limits_{n\in \mathbb{N}}\int_{B_n \setminus B_{n-1}} f(x) \, dx\\ &= \int_{\mR^2} f(x) \, dx.\qedhere
	\end{align*}
\end{proof}
The case of $\alpha \in (1,2)$ is included in the following result.
\begin{theorem}\label{th:strip}
	Let $\Omega = \mR \times (0,1)$. Assume that $K$ satisfies \AJ,\ \AD,\ \AT\ and $\sum_{n\geq 1} \int_{B(0,n)^c} K(0,x) \, dx < \infty$. Then the seminorms \eqref{eq:fullsemi} and \eqref{eq:truncated} are comparable.
\end{theorem}
\begin{proof}
	We split the domain $\Omega$ into open unit cubes $Q_n$ centered in $(n,1/2)$, $n\in \mathbb{Z}$, so that we have $\Omega \subseteq \bigcup\limits_{n\in\mathbb{Z}} \overline{Q_n}$. If we let $L_n = \mathrm{Int}[\overline{Q_{n-1}\cup Q_n\cup Q_{n+1}}]$, then $L_n$ is a uniform domain, hence by Theorem \ref{th:main}
	\begin{align*}
	\int_{L_n}\int_{L_n} (f(x) - f(y))^2 K(x,y) \red{\, dy \, dx}\approx \int_{L_n}\int_{B(x,\theta\delta(x))} (f(x) - f(y))^2 K(x,y) \red{\, dy \, dx}
	\end{align*}
	with the constant independent of $n$.
	Therefore for every $0<\theta\leq 1$
	\begin{align}
	&\int_{\Omega}\int_{B(x,\theta\delta(x))} (f(x) - f(y))^2 K(x,y) \red{\, dy \, dx}\nonumber\\ \approx  &\sum_{n\in\mathbb{Z}}\int_{L_n}\int_{L_n} (f(x) - f(y))^2 K(x,y)\, \red{dy \, dx}\nonumber\\
	\approx &\sum_{n\in\mathbb{Z}}\int_{Q_n}\int_{L_n} (f(x) - f(y))^2 K(x,y)\, \red{dy \, dx},\label{eq:kl}
	\end{align}
	so it suffices to show that the latter expression is comparable with the integral over $\Omega\times \Omega$.
	We have
	\begin{align*}
	&\int_{\Omega}\int_{\Omega} (f(x) - f(y))^2 K(x,y) \red{\, dy \, dx}\\ = &\sum_{i,j\in\mathbb{Z}} \int_{Q_i}\int_{Q_j} (f(x) - f(y))^2 K(x,y) \red{\, dy \, dx}\\
	\approx &\sum_{i \in \mathbb{Z}}\sum_{j + 1 < i} \int_{Q_i}\int_{Q_j}(f(x) - f(y))^2 K(x,y)\,\red{dy \, dx}\\ &+ \sum_{i \in \mathbb{Z}} \int_{Q_i}\int_{L_i}(f(x) - f(y))^2 K(x,y)\, \red{dy \, dx}.
	\end{align*}
	Clearly it suffices to estimate the first summand. Since the cubes are far apart, we have $|x-y| \approx |i-j|$ for $x\in Q_{i}$, $y\in Q_{j}$. Hence
	\begin{align}
	&\sum_{i \in \mathbb{Z}}\sum_{j + 1 < i} \int_{Q_i}\int_{Q_j}(f(x) - f(y))^2 K(x,y)\, \red{dy \, dx}\nonumber\\
	\lesssim\, &\red{\sum_{i \in \mathbb{Z}}\sum_{j + 1 < i} \int_{Q_{i}}\int_{Q_{j}} (f(x) - f_{Q_i})^2 K(x,y) \, dy \, dx}\label{eq:alfa}\\
	+&\red{\sum_{i \in \mathbb{Z}}\sum_{j + 1 < i} \int_{Q_{i}}\int_{Q_{j}} (f(y) - f_{Q_j})^2 K(x,y) \, dy \, dx}	\nonumber\\
	+&\sum_{i \in \mathbb{Z}}\sum_{j + 1 < i} \sum_{j\leq n < i}\int_{Q_i}\int_{Q_j} (f_{Q_{n+1}} - f_{Q_n})^2 |x-y|K(x,y) \, \red{dy \, dx}.\nonumber
	\end{align} 
	In this inequality we have used $(a_1 + \ldots + a_m)^2 \leq m(a_1^2 + \ldots +a_m^2)$ and $|Q_i| = |Q_j| = 1$.
	For the first term we use Jensen's inequality and the fact that the sum over $j$ is \red{uniformly bounded with respect to $i$ and $x\in Q_i$:}
	\begin{align*}
	&\sum_{i \in \mathbb{Z}} \int_{Q_i} (f(\red{x}) - f_{Q_i})^2\sum_{j + 1 < i}\int_{Q_j} K(x,y) \red{\, dy \, dx}\\ \lesssim &\sum_{i\in\mathbb{Z}}\int_{Q_i}\int_{Q_i} (f(y) - f(x))^2 \, \red{dy \, dx}.
	\end{align*}
	The latter expression does not exceed \eqref{eq:kl}. \red{The second term can be estimated in a similar way after changing the order of summation}.
	
	By Lemma \ref{lem:tech} the additional assumption on $K$ is equivalent to $$\sum_{n\geq 1}\int_{B(0,n)^c\cap \Omega} |x| K(0,x) \, dx < \infty.$$ We change the order of summation and use that fact to estimate the last term on the right hand side of \eqref{eq:alfa}:
	\begin{align*}
	&\sum_{i \in \mathbb{Z}}\sum_{j + 1 < i} \sum_{j\leq n < i} (f_{Q_{n+1}} - f_{Q_n})^2\int_{Q_i}\int_{Q_j} |x-y|K(x,y) \red{\, dy \, dx}\\
	= &\sum_{n\in\mathbb{Z}}(f_{Q_{n+1}} - f_{Q_n})^2\sum\limits_{i> n}\sum_{\substack{j+1<i\\j\leq n}}\int_{Q_i}\int_{Q_j} |x-y|K(x,y) \red{\, dy \, dx}\\
	\lesssim &\sum_{n\in\mathbb{Z}}(f_{Q_{n+1}} - f_{Q_n})^2 \leq \sum_{n\in\mathbb{Z}}\int_{Q_n}\int_{Q_{n+1}}(f(x) - f(y))^2 \red{\, dy \, dx}\\
	\lesssim &\sum_{n\in\mathbb{Z}}\int_{Q_n}\int_{Q_{n+1}}(f(x) - f(y))^2 \red{K(x,y)}\red{\, dy \, dx}.
	\end{align*}
\end{proof}

\begin{proof}[Proof of Theorem \ref{th:stripmd}]
	The idea is similar as above. We split $\Omega$ into a family of unit cubes $(Q_i)_{i\in \mathbb{Z}^k}$ and we let $L_i = \mathrm{Int}\big[\bigcup \{\overline{Q_j}\colon B(x_{Q_i},\sqrt{d}) \cap Q_j \ne \emptyset\}\big]$. By Theorem \ref{th:main}, for $0<\theta\leq 1$ we have
	\begin{align*}
	&\int_\Omega\int_{B(x,\theta\delta(x))} (f(x) - f(y))^2 |x-y|^{-d-\alpha} \, dy \, dx\\ \approx &\sum_{i\in \mathbb{Z}^k}\int_{Q_i}\int_{L_i} (f(x) - f(y))^2 |x-y|^{-d-\alpha} \, dy \, dx.
	\end{align*}
	For $i=(i_1,\ldots,i_k),\ j = (j_1,\ldots, j_k)$ and $m\in \mathbb{N}_0$, we say that $j>i+m$ if $j_1 > i_1+m,\ldots,\, j_k>i_k + m$. \red{By $j>m$ we mean $j>0+m$ and $j\geq i + m$ is defined by replacing \textit{all} the inequalities by weak ones.} By the radial symmetry of $|x-y|^{-d-\alpha}$ it suffices to show that under our assumptions on $l$ and $\alpha$ we have
	\begin{align*}
	&\sum_{i\in \mathbb{Z}^k}\int_{Q_i}\int_{L_i} (f(x) - f(y))^2 |x-y|^{-d-\alpha} \, dy \, dx\\ \gtrsim &\sum_{i\in \mathbb{Z}^k}\int_{Q_i}\sum_{j>i+1}\int_{Q_j} (f(x) - f(y))^2 |x-y|^{-d-\alpha} \, dy \, dx.
	\end{align*}
	In order to perform a decomposition similar to \eqref{eq:alfa} we fix a method of communication from $Q_i$ to $Q_j$, $j>i$: first we move \red{on the} coordinate $i_1$ until we reach $j_1$, and then we do the same with the next coordinates. The set of indexes of the cubes connecting $Q_i$ and $Q_j$ \red{in the way presented above, with $Q_i$ included and $Q_j$ excluded,} will be called $i\to j$. Note that $|i\to j| \approx |i-j|$. Let $\mathcal{N}(Q)$ be the successor of $Q$ on the way from $Q_i$ to $Q_j$. As before, we have $|i-j| \approx |x-y|$ for $x\in \red{Q_i}$, $y\in \red{Q_j}$, therefore
	\begin{align*}
	&\sum_{i\in \mathbb{Z}^k}\int_{Q_i}\sum_{j>i+1}\int_{Q_j} (f(x) - f(y))^2 |x-y|^{-d-\alpha} \, dy \, dx\\ \lesssim &\sum_{i\in \mathbb{Z}^k}\int_{Q_i}\sum_{j>i+1}\int_{Q_j} (f(x) - f_{Q_i})^2 |x-y|^{-d-\alpha} \, dy \, dx\\
	+ &\sum_{i\in \mathbb{Z}^k}\int_{Q_i}\sum_{j>i+1}\int_{Q_j} (f(y) - f_{Q_j})^2 |x-y|^{-d-\alpha} \, dy \, dx\\
	+ &\sum_{i\in \mathbb{Z}^k}\int_{Q_i}\sum_{j>i+1}\int_{Q_j}\sum_{n\in i\to j} (f_{Q_n} - f_{\mathcal{N}(Q_n)})^2 |x-y|^{-d-\alpha+1} \, dy \, dx.
	\end{align*}
	The first \red{two terms} can be handled as in the previous theorem. In the \red{latter} we change the order of summation and we get \red{that up to a constant it does not exceed}
	\begin{align*}
	\sum_{n\in\mathbb{Z}^k} \red{\bigg(\int_{L_n} |f_{Q_n} - f(\xi)|\, d\xi\bigg)^2} \sum_{j \geq n}\sum_{\substack{i\leq n\\i+1<j}}\int_{Q_{\red{i}}}\int_{Q_{\red{j}}} |x-y|^{-d-\alpha+1} \red{\, dy \, dx}.
	\end{align*}
	To finish the proof we note that the double sum over $i,j$ does not depend on $n$, hence we take $n=\red{(1,\ldots,1)}$ (for short, $n=1$) and \red{we estimate as follows:}
	\begin{align*}
	&\sum_{j\geq \red{1}}\sum_{\substack{i\leq \red{1}\\i+1<j}}\int_{Q_{\red{i}}}\int_{Q_{\red{j}}} |x-y|^{-d-\alpha+1} \, \red{dy \, dx}\\ \approx &\sum_{j\geq \red{1}} \int_{Q_{j}} \int_{B(y,|j|)^c\cap\Omega} |x-y|^{-d-\alpha+1} \, dx \, dy\\
	= &\sum_{j\geq \red{1}} \int_{Q_j} \sum_{m=0}^\infty \int_{(B(0,2^{m+1}|j|)\setminus B(0,2^m|j|))\cap\Omega} |x|^{-d-\alpha+1} \, dx \, dy\\ \approx &\sum_{j\geq \red{1}} \int_{Q_{j}} \sum_{m=0}^\infty (2^m|j|)^k(2^m|j|)^{-d-\alpha+1} \, dy\\
	\approx & \sum_{j\geq \red{1}} |j|^{\red{k-d-\alpha + 1}} = \sum_{j\geq \red{1}} |j|^{\red{-l-\alpha+1}} \approx \red{\sum_{j\in \mathbb{Z}^k\setminus \{0\}} |j|^{-l-\alpha+1},}
	\end{align*}
	which is finite provided that \red{$k-l - \alpha < -1$}.
\end{proof}
\section{Application: a new class of Markov processes}\label{sec:proces}
In this section we present how our comparability results can be applied to prove the existence of Markov stochastic processes corresponding to the truncated seminorms \eqref{eq:truncated}. Hereafter we work with Sobolev spaces, i.e. $p=q=2$.

\red{In this Section we will discuss several cases which depend on various results concerning Sobolev spaces and censored/reflected Markov processes, each with its own assumptions. Therefore we refrain from formulating any theorems here, as they would be unnecessarily complicated. Interested readers may gather the assumptions from the references provided to each case.}

\red{We will gradually introduce some notions concerning the Dirichlet forms in \textit{italics}, for details we refer to Fukushima, Oshima, and Takeda \cite[Chapter 1.1]{MR2778606}. Let $\mathcal E$ be a symmetric bilinear form with domain $D[\mathcal E] \subseteq L^2(\Omega)$ for some $\Omega\subseteq \mR^d$. Let $\mathcal E_1(u,u) = \mathcal E(u,u) + \|u\|_{L^2(\Omega)}^2$. We say that $(\mathcal{E},D[\mathcal E])$ (this pair will also be called form below) is \textit{closed} if, with respect to $\mathcal E_1$, every Cauchy sequence has a limit in $D[\mathcal E]$. We say that the form is \textit{closable} if it has a closed extension. In what follows we write
	$$\mathcal{E}^{\rm cen}(u,u) = \int_\Omega\int_\Omega (u(x) - u(y))^2 K(x,y) \, \red{dy \, dx},$$
	and for $\theta\in (0,1]$
	$$\mathcal{E}^{\rm tr} (u,u) = \int_\Omega \int_{B(x,\theta\delta(x))} (u(x) - u(y))^2 K(x,y) \, dy \, dx.$$
	The symbol $\mathcal E^{\rm cen}$ refers to the censored stable processes introduced by Bogdan, Burdzy, and Chen \cite{MR2006232}. There, the kernel was the one known from the fractional Sobolev spaces: $K(x,y) = c|x-y|^{-d-\alpha}$. Censored processes for more general $K$ corresponding to a class of subordinated Brownian motions were studied by Wagner \cite{Vanja}.}

\red{We will consider the above forms in two contexts. In the first one, we start with the space of smooth functions, compactly supported in $\Omega$: $C_c^\infty(\Omega)$. Using the arguments which follow equation (2.4) on \cite[page 93]{MR2006232} it can be shown that $(\mathcal E^{\rm cen},C_c^\infty(\Omega))$ is closable and \textit{Markovian} for arbitrary L\'evy kernel $K$ (in particular, any which satisfies \AJ) and set $\Omega$. If we let
	$$\mathcal{F} := \textrm{'Completion of }C^\infty_c(\Omega)\textrm{ with respect to }\mathcal{E}^{\rm cen}_1\,',$$
	then by \cite[Theorem 3.1.1]{MR2778606} $(\mathcal E^{\rm cen},\mathcal F)$ is closed and Markovian, that is, a \textit{Dirichlet} form. Furthermore, by construction, it is obvious that $C_c^\infty(D)$ is a \textit{core} for $(\mathcal E^{\rm cen},\mathcal F)$, hence the form is \textit{regular} and by \cite[Theorem 7.2.1]{MR2778606} to every regular Dirichlet form corresponds a Hunt process. Thus, when $\mathcal E^{\rm cen}$ is comparable to $\mathcal E^{\rm tr}$ we obtain the existence of a Hunt process with the Dirichlet form $(\mathcal E^{\rm tr},\mathcal F)$. We note that the arguments from \cite[page 93]{MR2006232} may be used directly with $\mathcal E^{\rm tr}$ because it has a similar structure. Then, independently of comparability results, we obtain a regular Dirichlet form $(\mathcal E^{\rm tr}, \mathcal F^{\rm tr})$, where
	$$\mathcal F^{\rm tr}  := \textrm{'Completion of }C^\infty_c(\Omega)\textrm{ with respect to }\mathcal{E}^{\rm tr}_1\,'.$$}
\red{\indent The second approach is by considering the domain corresponding to the active reflected form $\mathcal F^{\rm ref} := F_{2,2}(\Omega)$. Here the situation becomes more tedious, since in general $C_c^\infty(\Omega)$ (or even $C_c(\Omega)$) need not be dense in $\mathcal F^{\rm ref}$. However for some $K$ and $\Omega$ the density holds true, see e.g., \cite[Corollary~2.6]{MR2006232} and \cite[Corollary~2.9]{Vanja}. In that case we get that $\mathcal F = F_{2,2}(\Omega)$ and when the comparability holds, we in fact have $\mathcal F^{\rm tr} = F_{2,2} (\Omega)$. Then, the form $(\mathcal E^{\rm tr}, F_{2,2}(\Omega))$ is a regular Dirichlet form and there exists an associated Hunt process. If the density does not hold, the technical remedy is to change the reference set to $\overline \Omega$, cf. \cite[Remark 2.1]{MR2006232}. If $K$ and $\Omega$ are sufficiently regular, then there exist extension (and trace) operators between $F_{2,2}(\Omega)$ and $F_{2,2}(\mR^d)$, see e.g. Jonsson and Wallin \cite[Chapter V]{MR820626} or Rutkowski \cite[Section 6]{2018AR}. Thanks to them we may show that $C_c^\infty(\overline \Omega)$ is dense with respect to $\mathcal E^{\rm cen}_1$ in $F_{2,2}(\Omega)$ by using the results for the functions on the whole space $\mR^d$, available for very general L\'evy kernels, see, e.g., Bogdan, Grzywny, Pietruska-Pa\l{}uba, and Rutkowski \cite[Lemma A.5]{MR4088505} or Fiscella, Servadei, and Valdinoci \cite{MR3310082}. Then we obtain the existence of a process on $\overline \Omega$ corresponding to the regular Dirichlet form $(\mathcal E^{\rm cen}, F_{2,2}(\Omega))$ and the comparability yields the existence of the process corresponding to $(\mathcal E^{\rm tr}, F_{2,2}(\Omega))$.}

\red{The latter case seems more interesting in terms of applying the comparability results as we build regular Dirichlet forms from the truncated form $\mathcal E^{\rm tr}$ on the well-established Sobolev/Triebel--Lizorkin space $F_{2,2}(\Omega)$, which is then its natural domain.}
\bibliographystyle{abbrv}
\bibliography{TL}

\end{document}